\documentclass[reqno]{amsart}
\usepackage{amsmath}
\usepackage{amssymb}
\usepackage{amsfonts}
\usepackage{amsaddr}
\usepackage{graphicx}
\usepackage{booktabs}
\usepackage{graphicx}
\usepackage{subcaption}
\usepackage{arydshln}
\usepackage{color}
\usepackage{tikz}
\usepackage[noadjust]{cite}

\newtheorem{theorem}{Theorem}
\theoremstyle{plain}

\newtheorem{claim}{Claim}

\newtheorem{corollary}{Corollary}

\newtheorem{definition}{Definition}
\newtheorem{example}{Example}

\newtheorem{lemma}{Lemma}
\newtheorem{notation}{Notation}

\newtheorem{proposition}{Proposition}
\newtheorem{remark}{Remark}

\numberwithin{equation}{section}
\makeatletter
\renewcommand{\subsection}{\@startsection{subsection}{2}  \z@{.5\linespacing\@plus.7\linespacing}{.5\linespacing}  {\normalfont\bfseries}}
\makeatother

\begin{document}
\title[Gromov hyperbolicity of substitution graphs]{Gromov hyperbolicity of
substitution graphs}
\author[Qingcheng Zeng, Cheng Zeng, Yumei Xue, Huixia He]{%
    Qingcheng Zeng$^{\,\text{a}}$, 
    Cheng Zeng$^{\,\text{b}}$,     
    Yumei Xue$^{\,\text{a},*}$,    
    Huixia He$^{\,\text{a}}$       
}
\email{qczeng@buaa.edu.cn \textrm{(Q. Zeng)}}
\email{czeng@sdtbu.edu.cn \textrm{(C. Zeng)}}
\email{yxue@buaa.edu.cn \textrm{(Y. Xue)}}
\email{hehx@buaa.edu.cn \textrm{(H. He)}}

\address{\footnotesize\textsuperscript{\textnormal{a}}School of Mathematical Sciences, Beihang University, Beijing
100083, P. R. China}
\address{\footnotesize\textsuperscript{\textnormal{b}}School of Mathematics and Information Science, Shandong Technology
and Business University, Yantai, Shandong Province, 264003, P. R. China}

\subjclass[2020]{Primary 05C63; Secondary 53C23, 28A80}
\keywords{Infinite graphs, Substitution, Gromov hyperbolicity}
\thanks{*Corresponding author. }

\begin{abstract}
In this paper, we construct a class of infinite graphs, called substitution
graphs. The vertex set consists of all finite words over a finite alphabet.
A directed graph is formed by adding vertical edges connecting each word to
its children and horizontal edges defined recursively by two finite directed
graphs $\overrightarrow{G}$ and $\overrightarrow{J}$: edges among vertices
with the same parent follow $\overrightarrow{G}$, while edges between
vertices whose parents are horizontally linked follow $\overrightarrow{J}$.
The substitution graph is defined as its underlying graph. Substitution
graphs provide a purely combinatorial model of self-similar structures,
independent of any underlying geometric structure. Furthermore, we establish a necessary
and sufficient condition for substitution graphs to be hyperbolic,
formulated in terms of the vanishing of path matrices associated with
sufficiently long shortest horizontal paths. Based on this characterization,
we further derive several conditions that are either necessary or sufficient
for hyperbolicity, depending only on the generators $\overrightarrow{G}$ and 
$\overrightarrow{J}$.
\end{abstract}

\maketitle

\section{Introduction}

Graph substitution operations provide a fundamental mechanism for generating
infinite structures via the iteration of local substitution rules. Formal
frameworks for graph substitution were developed systematically by Previte 
\cite{Previte}, who introduced vertex--replacement schemes and iterated
substitution processes to analyze the asymptotic and dynamical behavior of
infinite graphs. Beyond the graph-theoretic setting, substitution schemes
arise naturally in symbolic dynamics, where they are central to spectral
theory and to recognizability issues for substitution subshifts \cite%
{Queffelec,Fogg,Mosse}. In geometric contexts, they provide standard
constructions of substitution tilings, including aperiodic examples and
tilings with fractal boundary behavior, and they serve as a framework for
analyzing inflation factors and hierarchical structure \cite%
{Mozes,Frank,Smilansky,Frettloh}. Analogous recursive constructions also
appear in models that exhibit self-similar or scale-free topology \cite%
{Chilakamarri,Li}.

Gromov hyperbolicity plays a central role in geometric group theory and the
large-scale geometry of graphs. A graph is hyperbolic if its geodesic
triangles are uniformly thin \cite{Gromov}. This condition captures negative
curvature at large scale and has many equivalent forms (see also \cite{Woess}
and Definition \ref{gromov} in this paper). Since Gromov's foundational work 
\cite{Gromov}, hyperbolicity has been extensively studied for Cayley graphs 
\cite{Meier} and random graphs \cite{Mitsche}.

Considerable attention has recently been directed toward establishing Gromov
hyperbolicity for various metrics in geometric function theory.\ For
example, the hyperbolicity of Riemann surfaces endowed with the Poincar\'{e}
metric has also been investigated in several works \cite%
{Hasto2006,Balogh,Hasto2010,Hasto2011}. In particular, several studies have
revealed that the hyperbolicity of a geodesic metric space is equivalent to
the hyperbolicity of a suitably associated graph; see, for instance, \cite%
{Rodriguez2004,Portilla,Rodriguez2007,Bermudo}. Hence, establishing criteria
for graph hyperbolicity is a natural next step, especially for structures
generated by substitution rules.

Many discrete structures of interest arise from iterated function systems
(IFSs), which generate self-similar sets \cite{Hutchinson,Falconer} and
associated hierarchical graphs \cite{Bao}. Lau and Wang \cite%
{Lau2009,Lau2017} demonstrated that for a contractive IFS, the symbolic
space naturally admits a hyperbolic graph structure (augmented tree)
reflecting the relationship among neighboring cells, whose hyperbolic
boundary with the Gromov metric is H\"{o}lder equivalent to the attractor.
This line of inquiry originates from Kaimanovich's geometric interpretation 
\cite{Kaimanovich} of the Martin boundary for Markov chains on the Sierpi%
\'{n}ski gasket, a problem initially investigated by Denker and Sato \cite%
{Denker}. This geometric correspondence has proved instrumental in
subsequent studies, including the analysis of probabilistic potential theory
on fractals \cite{Kong2017,Kong-arXiv} and the classification of
bi-Lipschitz equivalence \cite{Luo}. More recently, Kong, Lau and Wang \cite%
{Kong2021} introduced a broad class of hyperbolic expansion graphs, thereby
capturing the essential properties from the augmented trees and the
hyperbolic boundaries.

In this paper, we introduce a class of infinite graphs, termed \textbf{%
substitution graphs}. Given a finite alphabet $\Sigma $ of size $N$, one
naturally obtains a rooted tree whose vertices are all finite words over $%
\Sigma $, with vertical edges connecting each word to its children.
Horizontal edges are then added recursively according to two finite directed
graphs $\overrightarrow{G}$ and $\overrightarrow{J}$, where $\overrightarrow{%
G}$\ is a directed simple graph on $N$ vertices corresponding to the letters
of\ $\Sigma $, and $\overrightarrow{J}$ is a subgraph of the complete
directed bipartite graph between two disjoint copies of $\Sigma $. Copies of 
$\overrightarrow{G}$ connect vertices with the same parent, while copies of $%
\overrightarrow{J}$ create horizontal edges between vertices whose parents
are already connected. The substitution graph is defined as the underlying
graph of the directed graph formed by the vertical and horizontal edges. A
precise definition of substitution graphs is given in Section 2. This
resulting structure serves as a purely combinatorial analogue to the
augmented trees associated with IFS. However, unlike the augmented trees and
expansion graphs arising from IFS, our construction is independent of any
underlying geometric structure. As a result, the substitution graphs may
have richer horizontal connections and need not be tree-like globally. We
aim to determine when such substitution graphs are hyperbolic and how this
depends on the generators $\overrightarrow{G}$ and $\overrightarrow{J}$.

Our first main result is a complete characterization of hyperbolicity for
substitution graphs. For each shortest horizontal path we associate a path
matrix encoding the substitution of the path under $\overrightarrow{G}$ and $%
\overrightarrow{J}$. Roughly speaking, we prove that a substitution graph $%
\mathcal{U}$ is hyperbolic if and only if all sufficiently long shortest
horizontal paths have path matrices that vanish (see Theorem \ref{th-equ}).
Based on this characterization, we derive effective sufficient and necessary
conditions on the generators $\overrightarrow{G}$ and $\overrightarrow{J}$
that either ensure or rule out the hyperbolicity of $\mathcal{U}$, which can
be verified in finitely many steps. In particular, in Theorem \ref{th-degree}
we show that hyperbolicity is guaranteed if the biadjacency matrix $K$ of $%
\overrightarrow{J}$ is nilpotent and the coupling graph has uniformly
bounded in-degree or out-degree, where the coupling graph is obtained by
gluing two copies of $\overrightarrow{G}$ along a copy of $\overrightarrow{J}
$ (see Definition \ref{def-coup} for more details); in Theorem \ref{th-lac}
we establish a hyperbolicity criterion based on the nilpotency of $K$ and a
local arborescence condition combined with the acyclicity of the underlying
coupling graph, which rigidly constrains the behavior of shortest horizontal
paths.

The paper is organized as follows. Section 2 introduces substitution graphs,
together with illustrative examples. Section 3 reviews Gromov hyperbolicity
and presents preliminary results. Section 4 establishes the main
characterization theorem and develops necessary and sufficient conditions
for hyperbolicity. Section 5 illustrates the results by analyzing the
examples introduced earlier.

\bigskip

\section{Substitution graphs}

For a finite alphabet $\Sigma =\{1,2,\ldots ,N\},$ there exists a natural 
\emph{rooted tree} structure $(V,\overrightarrow{E}_{v})$. The vertex set $V$
is the set of all finite sequences over the alphabet $\Sigma ,$ i.e. 
\begin{equation*}
V=\bigcup_{n=0}^{\infty }\Sigma _{n},\ \Sigma _{n}=\Sigma ^{n}\text{ }(\text{%
convention }\Sigma _{0}=\{\varnothing \}).
\end{equation*}%
The set of \emph{vertical edges} is defined as 
\begin{equation*}
\overrightarrow{E}_{v}=\{(\mathbf{i},\mathbf{i}\sigma ):\mathbf{i}\in
V,\sigma \in \Sigma \}.
\end{equation*}%
Then $(V,\overrightarrow{E}_{v})$ is a directed rooted tree with $\vartheta
=\varnothing $ as the root.

Note that in a rooted tree, every vertex can have multiple children but
every non-root vertex has exactly one parent. For a non-empty word $u\in V,$
we use $u^{[-k]}$ to denote the unique $k$-th ancestor of $u.$ For
simplicity of notation, we also denoted by $u^{-}$ the parent ($1$-th
ancestor) of $u$. We denote by $\varphi (u)$ the last letter of $u.$

Let $\Sigma ^{\prime }=\{1^{\prime },2^{\prime },\ldots ,N^{\prime }\}$ be a
disjoint copy of $\Sigma ,$ where $i^{\prime }\in \Sigma ^{\prime }$
corresponds to $i\in \Sigma .$ We define the \emph{complete directed
bipartite graph }from $\Sigma $ to $\Sigma ^{\prime }$, denoted by $%
\overrightarrow{K}_{N,N}$, as the graph with vertex set $\Sigma \cup \Sigma
^{\prime }$ and edge set consisting of all directed edges from $\Sigma $ to $%
\Sigma ^{\prime }.$ Formally, the edge set of $\overrightarrow{K}_{N,N}$ is
given by%
\begin{equation*}
E(\overrightarrow{K}_{N,N})=\{(i,j^{\prime }):i\in \Sigma ,\text{ }j^{\prime
}\in \Sigma ^{\prime }\}.
\end{equation*}

\begin{definition}
Let $(V,\overrightarrow{E}_{v})$ be the natural rooted tree generated by $%
\Sigma =\{1,2,\ldots ,N\}.$ Given a directed simple graph $\overrightarrow{G}
$ on $\Sigma \ $and a subgraph $\overrightarrow{J}$ of $\overrightarrow{K}%
_{N,N}$, we recursively define the set of \emph{horizontal edges} $%
\overrightarrow{E}_{h}\ $as follows. For any $n\geq 1$ and any two distinct
vertices $u,v\in \Sigma _{n},$ $(u,v)\in \overrightarrow{E}_{h}$ if either

\begin{enumerate}
\item[(i)] $u^{-}=v^{-}$ and $(\varphi (u),\varphi (v))\in E(\overrightarrow{%
G}),$ or

\item[(ii)] $(u^{-},v^{-})\in \overrightarrow{E}_{h}$ and $(\varphi
(u),\varphi (v)^{\prime })\in E(\overrightarrow{J}).$
\end{enumerate}

Let $\overrightarrow{E}=\overrightarrow{E}_{v}\cup \overrightarrow{E}_{h}$
and $\mathcal{D}=(V,\overrightarrow{E})$. The \emph{substitution graph},
denoted by $\mathcal{U}=(V,E)$, is defined as the \emph{underlying graph} of 
$\mathcal{D}$.
\end{definition}

\begin{notation}
The substitution graph $\mathcal{U}$ is generated by $\overrightarrow{G}$
and $\overrightarrow{J}$, we also use the notation $(\overrightarrow{G},%
\overrightarrow{J})$ for $\mathcal{U}$ when emphasizing its generating
components. Let $G$ and $J$ denote the underlying graphs of $\overrightarrow{%
G}$ and $\overrightarrow{J}$, respectively.

For a horizontal edge $(u,v)\in \vec{E}_{h}$, we say that $(u,v)$ is \emph{%
of $G$--type} if it is produced by substitution\emph{\ }rule~(i), i.e.%
\begin{equation*}
u^{-}=v^{-}\text{ and }(\varphi (u),\varphi (v))\in E(\vec{G}),
\end{equation*}%
and that $(u,v)$ is \emph{of $J$--type} if it is produced by substitution
rule~(ii), i.e. 
\begin{equation*}
(u^{-},v^{-})\in \vec{E}_{h}\text{ and }(\varphi (u),\varphi (v)^{\prime
})\in E(\vec{J}).
\end{equation*}%
This classification is unique for each horizontal edge.
\end{notation}

Note that if $N=1,$ then the substitution graph $\mathcal{U}$\ degenerates
to a ray (a one-way infinite path), and if $E(G)=\emptyset $, then $\mathcal{%
U}$ is a tree, which is trivially hyperbolic. To avoid such trivial cases,
we assume $N\geq 2$ and $E(G)\neq \emptyset $\ throughout the remainder of
this paper.

We illustrate the definition by four examples, see Fig. \ref{Examples}.
Their hyperbolicity will be analyzed in Section 4.4 using the criteria
established later.

\begin{example}[Levelwise Arborescent Graph]
\label{exa1}Let $\overrightarrow{G}$ be a directed path of length $2$ (i.e. $%
1\rightarrow 2\rightarrow 3$) and $E(\overrightarrow{J})=\{(2,1^{\prime })\}$%
. Each level of the substitution graph is a Arborescent.
\end{example}

\begin{example}
\label{exa2}Let $\overrightarrow{G}$ be a directed graph with vertices $%
\{1,2,3,4\}$ and edges $1\rightarrow 2,$ $1\rightarrow 3,$ $4\rightarrow 3,$
and let $E(\overrightarrow{J})=\{(2,1^{\prime })\}$. Here, $\overrightarrow{G%
}$ satisfies the local arborescence condition (see Definition \ref{LAC}), a
property relevant for checking the hyperbolicity of the substitution graph.
\end{example}

\begin{example}[Levelwise Torus Graph]
\label{exa3}Let $\overrightarrow{G}$ be a directed 3-cycle and $%
\overrightarrow{J}$ be the identity mapping. The level-$n$ horizontal graph
corresponds to a discrete torus structure.
\end{example}

\begin{example}[Cyclic Inflation Graph]
\label{exa4}Let $\overrightarrow{G}$ be a directed 4-cycle and $E(%
\overrightarrow{J})=\{(4,1^{\prime })\}$.
\end{example}

\begin{figure}[tbph]
\centering
\par
\begin{subfigure}[t]{0.48\textwidth}
        \centering
        \includegraphics[height=2cm]{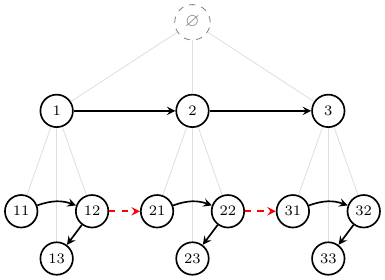} 
        \caption{Levelwise arborescent graph (Ex. 1)}
    \end{subfigure}\hfill 
\begin{subfigure}[t]{0.48\textwidth}
        \centering
        \includegraphics[height=2cm]{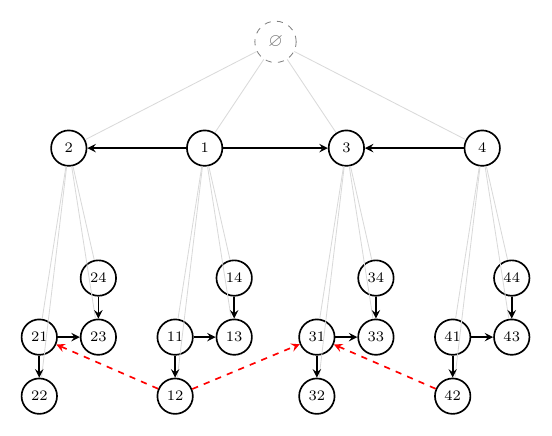} 
        \caption{Ex. 2}
    \end{subfigure}
\par
\vspace{1ex}
\par
\begin{subfigure}[t]{0.48\textwidth}
        \centering
        \includegraphics[height=2cm]{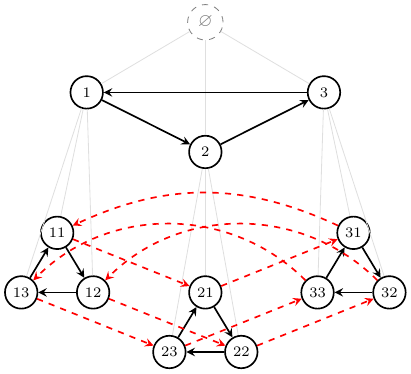} 
        \caption{Levelwise torus graph (Ex. 3)}
    \end{subfigure}\hfill 
\begin{subfigure}[t]{0.48\textwidth}
        \centering
        \includegraphics[height=2cm]{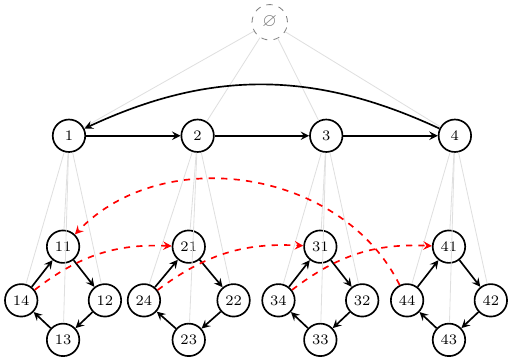} 
        \caption{Cyclic inflation graph (Ex. 4)}
    \end{subfigure}
\caption{Examples of substitution graphs.parent}
\label{Examples}
\end{figure}

\bigskip

\section{Gromov hyperbolicity and preliminaries}

\subsection{Gromov hyperbolicity and known results}

Throughout this section, we consistently assume that the graph $\Gamma
=(V,E) $ is a \emph{rooted graph}, i.e., a locally finite connected
undirected graph with a fixed root vertex $\vartheta \in V$. A finite
sequence of vertices 
\begin{equation*}
p=(u_{0},u_{1},\dots ,u_{\ell })
\end{equation*}%
is called a \emph{walk} if $(u_{q-1},u_{q})\in E$ for all $q=1,\dots ,\ell $%
. The length of $p$ is $L(p):=\ell $. If, in addition, all edges $%
(u_{q-1},u_{q})$ are distinct, then $p$ is a \emph{path}. Let $d(u,v)$
denote the \emph{graph distance} between $u$ and $v$, i.e., the length of
the shortest path from $u$ to $v$ if such path exists ($d(u,v)=\infty $
otherwise).\ A path between two vertices $u$ and $v$ is called a \emph{%
geodesic} if its length equals the graph distance $d(u,v)$. We use $%
\left\vert u\right\vert =d(\vartheta ,u)$ to denote the distance from the
root to the vertex $u$. For any integer $m\geq 0$ and $v\in V$, let 
\begin{equation*}
J_{m}(v):=\{u\in V:\left\vert u\right\vert =\left\vert x\right\vert
+m\},\,J_{-m}(v):=\{w\in V:v\in J_{m}(w)\}
\end{equation*}%
be the $m$-th descendant set and the $m$-th predecessor set of $v$
respectively.

Let $E_{v}=\{(u,v)\in E:\left\vert u\right\vert -\left\vert v\right\vert
=\pm 1\}$ and $E_{h}=\{(u,v)\in E:\left\vert u\right\vert =\left\vert
v\right\vert \}$ denote the vertical edge set and the horizontal edge set
respectively. Clearly $E=E_{v}\cup E_{h}$ and $E_{v}\cap E_{h}=\emptyset $.
A \emph{horizontal path }is a finite sequence of vertices 
\begin{equation*}
p=(u_{0},u_{1},\cdots ,u_{\ell })
\end{equation*}%
such that $(u_{i},u_{i+1})\in E_{h}$ for every $0\leq i<\ell $. In
particular, if there exists $n\geq 0$ such that $\left\vert u_{i}\right\vert
=n$ for all $i$, then we say that $p$ is a horizontal path at level $n$. For
vertices $u,v\in V$ with $\left\vert u\right\vert =\left\vert v\right\vert $%
, a \emph{shortest horizontal path} from $u$ to $v$ is a horizontal path 
\begin{equation*}
(u_{0},u_{1},\cdots ,u_{\ell }),\text{ }u_{0}=u,\text{ }u_{\ell }=v,
\end{equation*}%
whose length $\ell $ (the number of edges) is minimal among all horizontal
paths from $u$ to $v$.

We define the \emph{horizontal distance }$d_{h}(\cdot ,\cdot )$ to be the
graph distance in the horizontal graph $(V,E_{h}).$ It follows immediately
that $d_{h}(x,y)=\infty $ for $\left\vert x\right\vert \neq \left\vert
y\right\vert ,$ and $d(x,y)\leq d_{h}(x,y)$ for all $x,y\in V$. When the
equality $d(x,y)=d_{h}(x,y)$ holds, there exists a geodesic that lies in $%
(V,E_{h})$, which we call a \emph{horizontal geodesic} of $\Gamma $.

\begin{definition}[\protect\cite{Gromov}]
\label{gromov}Suppose $\Gamma $ is a rooted graph, the \emph{Gromov product}
of $u,v\in V(\Gamma )$ is defined by 
\begin{equation*}
\left\vert u\wedge v\right\vert :=\frac{1}{2}(\left\vert u\right\vert
+\left\vert v\right\vert -d(u,v)).
\end{equation*}%
We call $\Gamma $ is \emph{(Gromov) }$\delta $\emph{-hyperbolic} if there
exists $\delta \geq 0$ such that 
\begin{equation*}
\left\vert u\wedge v\right\vert \geq \min \{\left\vert u\wedge w\right\vert
,\left\vert w\wedge v\right\vert \}-\delta \text{ for all }u,v,w\in V(\Gamma
).
\end{equation*}
\end{definition}

\begin{definition}[\protect\cite{Kong2021}]
A rooted graph $\Gamma =(V,E)$ is said to be an \emph{expansive graph} if
for $x,y\in V,$ 
\begin{equation*}
d_{h}(x,y)>1\Rightarrow d_{h}(u,v)>1,\text{ }\forall u\in J_{1}(x),\,v\in
J_{1}(y).
\end{equation*}%
Let $m,k$ be two positive integers. We call $\Gamma $ is $(m,k)$\emph{%
-departing} if for $x,y\in V,$%
\begin{equation*}
d_{h}(x,y)>k\Rightarrow d_{h}(u,v)>2k,\text{ }\forall u\in J_{m}(x),\,v\in
J_{m}(y).
\end{equation*}
\end{definition}

An \emph{augmented (rooted) tree} $\Gamma =(V,E)$ (as defined in [14, 26,
28]) is an expansive graph in which the vertical graph $(V,E_{v})$ is a
tree. The following theorem provide a useful criterion for the hyperbolicity
of the expansive graphs.

\begin{theorem}[\protect\cite{Kong2021}]
\label{th-Kong}Let $\Gamma \ $be an expansive graph. The following
assertions are equivalent.

\begin{enumerate}
\item[(i)] $\Gamma $ is hyperbolic;

\item[(ii)] $\exists L<\infty $ such that the lengths of all horizontal
geodesics are bounded by $L$;

\item[(iii)] $\Gamma $ is $(m,k)$-departing for some positive integers $m$
and $k.$
\end{enumerate}
\end{theorem}

\bigskip

\subsection{Preliminaries}

To establish the criteria for the hyperbolicity of substitution graph, we
investigate the algebraic structure of the semigroup generated by a $0$-$1$
matrix $A$ and its transpose $A^{\mathrm{T}}$. The existence of zero
elements within the generated semigroup plays a key role in our subsequent
analysis.

Let $A\in M_{n}(\mathbb{R})$. The semigroup generated by $A$ and $A^{\mathrm{%
T}}$ is defined as 
\begin{equation*}
\left\langle A,A^{\mathrm{T}}\right\rangle =\{A_{1}A_{2}\cdots A_{k}:k\in 
\mathbb{N},A_{i}\in \{A,A^{\mathrm{T}}\}\}.
\end{equation*}%
The following lemmas investigate the vanishing of elements in the semigroup.

\begin{lemma}
\label{le-A1Ak}Let $A$ be a $n\times n$ $0$-$1$ matrix. If $A$ is not
nilpotent, then every element in the semigroup $\left\langle A,A^{\mathrm{T}%
}\right\rangle $ is non-zero.
\end{lemma}

\begin{proof}
Let $\Gamma $ denote the directed graph associated with$\ A$ whose vertex
set is $\{1,\ldots ,n\}$ and in which there is a directed edge $i\rightarrow
j$ if and only if $A_{i,j}=1$. The non-nilpotency of $A$ implies that $%
\Gamma $ contains a directed cycle. Let $S$ be a strongly connected
component of $\Gamma $ that contains a directed cycle, then the subgraph
induced by $S$ is a strongly connected finite directed graph with at least
one edge. In particular, every vertex in $S$ has both in-neighbour and
out-neighbour in $S$.

We claim that for any every integer $k\geq 1$ and any sequence of directions 
$d_{1}\cdots d_{k}\in \{0,1\}^{k},$ there exist vertices $v_{0},\ldots
,v_{k} $ in $S$ such that for each $t=1,\ldots ,k$%
\begin{eqnarray*}
d_{t} =0&\Longrightarrow& (v_{t-1},v_{t})\in E(\Gamma ) \\
d_{t} =1&\Longrightarrow& (v_{t},v_{t-1})\in E(\Gamma ).
\end{eqnarray*}%
Choose any $v_{0}\in S$. For $t=1$, if $d_{1}=0$ pick $v_{1}$ to be an
out-neighbour of $v_{0}$ in $S$, and if $d_{1}=1$ pick $v_{1}$ to be an
in-neighbour of $v_{0}$ in $S$. Proceed inductively. Having chosen $%
v_{0},\dots ,v_{t-1}\in S$, if $d_{t}=0$ choose $v_{t}$ to be an
out-neighbour of $v_{t-1}$ in $S$, and if $d_{t}=1$ choose $v_{t}$ to be an
in-neighbour of $v_{t-1}$ in $S$. The required edges exist by the above
property of $S$, and all $v_{t}$ remain in $S$. This proves the claim.

Now let $M=A_{1}A_{2}\cdots A_{k}\in \left\langle A,A^{\mathrm{T}%
}\right\rangle .$ Define a direction sequence by 
\begin{equation*}
d_{t}=\left\{ 
\begin{array}{ll}
0, & \text{if }A_{t}=A, \\ 
1, & \text{if }A_{t}=A^{\mathrm{T}},%
\end{array}%
\right. \text{for all }1\leq t\leq k.
\end{equation*}

Applying the claim, choose vertices $v_{0},\dots ,v_{k}\in S$ with the
corresponding edge property. Then, for each $t=1,\dots ,k$ we have 
\begin{equation}
(A_{t})_{v_{t-1},v_{t}}=1,  \label{A_t=1}
\end{equation}%
because if $A_{t}=A$ and $(v_{t-1},v_{t})\in E(\Gamma )$, then $%
A_{v_{t-1},v_{t}}=1$, while if $A_{t}=A^{\mathrm{T}}$ and $%
(v_{t},v_{t-1})\in E(\Gamma )$, then $A_{v_{t-1},v_{t}}^{\mathrm{T}%
}=A_{v_{t},v_{t-1}}=1$. By the definition of matrix multiplication and (\ref%
{A_t=1}), we have%
\begin{equation*}
M_{v_{0},v_{k}}=(A_{1}A_{2}\cdots A_{k})_{v_{0},v_{k}}\geq
\prod_{t=1}^{k}(A_{t})_{v_{t-1},v_{t}}=1.
\end{equation*}%
Hence $M$ is non-zero. Since $M$ was arbitrary, every element of $\langle
A,A^{\mathrm{T}}\rangle $ is non-zero.
\end{proof}

Let $A\in M_{n}(\mathbb{R)}$ be a $0$-$1$ nilpotent matrix. We associate
with $A$ a poset $(P,\leq )$ on the set $P=\{{1,\ldots ,n\}}$ whose covering
relation $j$ covers $i$ holds precisely when $A_{i,j}=1.$

\begin{definition}
Let $(P,\leq )$ be a finite poset. We say that $P$ is \emph{graded of rank }$%
\emph{r}$ if there exists a rank function 
\begin{equation*}
\rho :P\longrightarrow \{0,1,\dots ,r\}
\end{equation*}%
such that

\begin{enumerate}
\item[(i)] every minimal element of $P$ has rank $0$;

\item[(ii)] whenever $y$ covers $x$ in $P$, one has $\rho (y)=\rho (x)+1$;

\item[(iii)] the image of $\rho $ is exactly $\{0,1,\dots ,r\}$, i.e. every
maximal chain in $P$ has length $r$.
\end{enumerate}
\end{definition}

\begin{lemma}
\label{le-SN}Let $A\in M_{n}(\mathbb{R})$ be a nilpotent $0$-$1$ matrix with
nilpotency index $I$ and $M=A_{1}\cdots A_{k}\in \left\langle A,A^{\mathrm{T}%
}\right\rangle $ with each $A_{i}\in \{A,A^{\mathrm{T}}\}$. Suppose the
poset $(P,\leq )$ associated with $A$ is graded. Then $M=\boldsymbol{0}$ if
and only if there exists a contiguous subproduct $M^{\prime }:=A_{p}\cdots
A_{q}$ $(1\leq p<q\leq k)$ such that 
\begin{equation*}
\left\vert N_{0}(M^{\prime })-N_{1}(M^{\prime })\right\vert \geq I,
\end{equation*}%
where $N_{0}$ and $N_{1}$ denote the counts of $A$ and $A^{\mathrm{T}}$ in
the subproduct, respectively.
\end{lemma}

\begin{proof}
Since $A$ has nilpotency index $I,$ then $(P,\leq )$ is graded of rank $I-1$%
. Then there exists a rank function $\rho :P\longrightarrow \{0,1,\dots ,r\}$
such that for any indices $i,j,$ 
\begin{equation*}
A_{i,j}\neq 0\Longrightarrow \rho (j)=\rho (i)+1,\text{ }A_{i,j}^{\mathrm{T}%
}\neq 0\Longrightarrow \rho (j)=\rho (i)-1.
\end{equation*}%
Consequently, multiplication by $A$ (resp. $A^{\mathrm{T}}$) strictly
increase (resp. decreases) the rank of the indices corresponding to non-zero
entries.

For any contiguous subproduct $M^{\prime }:=A_{p}\cdots A_{q},$ define 
\begin{equation*}
\Delta (M^{\prime }):=N_{0}(M^{\prime })-N_{1}(M^{\prime }).
\end{equation*}%
If $(M^{\prime })_{i,j}\neq 0,$ then there exists a sequence of indices
connecting $i$ to $j$ through the non-zero entries of the factors, implying%
\begin{equation*}
\rho (j)-\rho (i)=\Delta (M^{\prime }).
\end{equation*}%
Since the image of $\rho $ is $\{0,1,\dots ,I-1\},$ any such non-zero entry
requires 
\begin{equation}
\left\vert \Delta (M^{\prime })\right\vert =\left\vert \rho (j)-\rho
(i)\right\vert \leq I-1.  \label{Delta<I-1}
\end{equation}

If there exists a subproduct $M^{\prime }$ with $\left\vert \Delta
(M^{\prime })\right\vert \geq I,$ then (\ref{Delta<I-1}) fails and thus $%
M^{\prime }=\boldsymbol{0},$ whence $M=\boldsymbol{0}$.

Conversely, suppose $M=\boldsymbol{0}$. If (\ref{Delta<I-1}) were valid for
all contiguous subproducts, then the associated rank sequence would remain
within the admissible range $\{0,1,\dots ,I-1\}$ at every step. This imply
the existence of at least one sequence of indices corresponding to non-zero
entries, yielding $M\neq 0,$ a contradiction. Therefore, there must be some
subproduct $M^{\prime }$ satisfying $\left\vert \Delta (M^{\prime
})\right\vert \geq I.$
\end{proof}

The gradedness assumption on the associated poset is necessary. If the poset
is not graded, then the counting argument in the lemma may fail. This is
illustrated by the following example.

\begin{example}
Let%
\begin{equation*}
A=%
\begin{pmatrix}
0 & 1 & 0 & 0 & 0 & 0 \\ 
0 & 0 & 1 & 0 & 0 & 0 \\ 
0 & 0 & 0 & 0 & 0 & 0 \\ 
0 & 0 & 1 & 0 & 1 & 0 \\ 
0 & 0 & 0 & 0 & 0 & 1 \\ 
0 & 0 & 0 & 0 & 0 & 0%
\end{pmatrix}%
,
\end{equation*}%
then $A$ is a nilpotent $0$--$1$ matrix with nilpotency index $I=3$.
Consider the poset $(P,\leq )$ associated with $A.$ In this poset the
element $3$ lies on two maximal chains of different lengths%
\begin{equation*}
1<2<3,\text{ }4<3,
\end{equation*}%
so $(P,\leq )$\ is not graded. For $M=A^{2}A^{\mathrm{T}}A^{2},$ we have 
\begin{equation*}
\left\vert N_{0}(M)-N_{1}(M)\right\vert \geq I\text{ but }M\neq \mathbf{0}.
\end{equation*}
\end{example}

\begin{remark}
\label{rmk0}If a element in $\langle A,A^{\mathrm{T}}\rangle $ is strictly
alternating in $A$ and $A^{\mathrm{T}}$, then it is non-zero whenever $A\neq 
\mathbf{0}$. Hence a vanishing product cannot be produced by a strictly
alternating word. This follows from the fact that any such word is of the
form $(AA^{\mathrm{T}})^{m}$, $(A^{\mathrm{T}}A)^{m}$, $(AA^{\mathrm{T}%
})^{m}A$, or $(A^{\mathrm{T}}A)^{m}A^{\mathrm{T}}$, and $AA^{\mathrm{T}}$
and $A^{\mathrm{T}}A$ are non-zero positive semidefinite matrices.
\end{remark}

\bigskip

\section{Hyperbolicity of substitution graphs}

We now turn to the study of the hyperbolicity of substitution graphs.
Throughout this section, we fix $\mathcal{U}=(V,E)$ as the substitution
graph generated by $\overrightarrow{G}$ and $\overrightarrow{J},$ where $%
\overrightarrow{G}$ is defined on $\Sigma =\{1,\cdots ,N\}.$ The
construction immediately shows that $\mathcal{U}$ is an augmented tree, and
therefore an expansive graph. Note that if $E(J)=\emptyset $, then
horizontal edges occur only between vertices having the same parent, and
there are no $\overrightarrow{J}$-type horizontal edges. In this case every
horizontal geodesic has length at most the diameter of $G$. Hence, by
Theorem \ref{th-Kong}, substitution graph $\mathcal{U}$ is hyperbolic. To
avoid such trivial situations, in the sequel we assume $N\geq 2$, $E(G)\neq
\emptyset $ and $E(J)\neq \emptyset $.

Utilizing the hyperbolicity criteria for expansive graphs (Theorem \ref%
{th-Kong}), we provide a characterization of the hyperbolicity of $\mathcal{U%
}$ as follows (Theorem \ref{th-equ}).

\subsection{General characterization of hyperbolicity}

We begin by introducing two classes of matrices that play a key role in
characterizing the hyperbolicity of substitution graphs $\mathcal{U}=(%
\overrightarrow{G},\overrightarrow{J})$.

Let $K=K(\overrightarrow{J})$ denote the \emph{biadjacency matrix} of $%
\overrightarrow{J}$, i.e., $K$ is a $0$-$1$ matrix and 
\begin{equation*}
K_{i,j}=1\text{ if and only if }(i,j^{\prime })\in E(\overrightarrow{J}).
\end{equation*}%
We also write $K_{0}=K$ and $K_{1}=K^{\mathrm{T}}$.

Let $p=(u_{0},u_{1},\cdots ,u_{\ell })$ be a horizontal path in $\mathcal{U}$%
. We may also regard $p$ as a directed path in the directed graph $\mathcal{D%
}$. The direction sequence of $p$ is defined as $\mathfrak{d}%
(p)=d_{1}d_{2}\cdots d_{\ell },$ where for each $q=1,\ldots ,\ell ,$ 
\begin{equation*}
d_{q}=\left\{ 
\begin{array}{cc}
0, & \text{if }(u_{q-1},u_{q})\in \overrightarrow{E}_{h}, \\ 
1, & \text{if }(u_{q},u_{q-1})\in \overrightarrow{E}_{h}.%
\end{array}%
\right.
\end{equation*}%
The \emph{path matrix} associated with $p$ is defined by 
\begin{equation*}
M(p)=\prod_{i=1}^{\ell }K_{d_{i}}=K_{d_{1}}K_{d_{2}}\cdots K_{d_{\ell }}.
\end{equation*}%
Note that if $(i,j^{\prime })\in E(\overrightarrow{J})$, then $%
K_{i,j}=1=K_{j,i}^{\mathrm{T}}$. Consequently, if $(x,y)\in E_{h}$ is of $J$%
-type, then either $(\varphi (x),\varphi (y)^{\prime })\in E(\overrightarrow{%
J})$ or $(\varphi (y),\varphi (x)^{\prime })\in E(\overrightarrow{J}),$ and
hence%
\begin{equation}
(K_{d})_{\varphi (x),\varphi (y)}=1,  \label{K_ij=1}
\end{equation}%
where $d=\mathfrak{d}(x,y)$ is the direction of $(x,y).$

Let $p=(x_{0},x_{1},\cdots ,x_{t})$ be a horizontal path in $\mathcal{U}$
with length $t\geq 1$. Consider the sequence of parents $%
\{x_{0}^{-},x_{1}^{-},\cdots ,x_{t}^{-}\}.$\ The property of augmented tree
implies that either $x_{i}^{-}=x_{i+1}^{-}$ or $(x_{i}^{-},x_{i+1}^{-})\in
E_{h}$.\ By removing consecutive repetitions from the parent sequence, we
obtain a horizontal walk (called projected walk of $p$) 
\begin{equation*}
p^{-}:=(y_{0},y_{1},\cdots ,y_{\ell })
\end{equation*}%
joining $x_{0}^{-}$ and $x_{t}^{-}$, where 
\begin{equation*}
y_{0}=x_{0}^{-},\text{ }y_{\ell }=x_{t}^{-}\text{ and }y_{i}\in
\{x_{0}^{-},x_{1}^{-},\cdots ,x_{t}^{-}\},\text{ }i=0,1,\cdots ,\ell .
\end{equation*}%
Since each edge in $p$ contributes at most $1$ to the length of $p^{-}$ (and 
$0$ if the parents are identical), we clearly have $\ell \leq t.$ This
construction implies the following expansive property%
\begin{equation}
d_{h}(u,v)\geq d_{h}(u^{-},v^{-})  \label{expansive}
\end{equation}%
for any vertices $u$ and $v$ at the same level.

\begin{lemma}
\label{le-M}Let $u,v\in \Sigma _{n}$ be distinct vertices for some $n\in 
\mathbb{N}$ and set $i=\varphi (u),j=\varphi (v).$ Then $%
d_{h}(u,v)=d_{h}(u^{-},v^{-})$ if and only if there exists a shortest
horizontal path $p$ from $u^{-}$ to $v^{-}$ such that $M(p)_{i,j}>0$.
\end{lemma}

\begin{proof}
Suppose first that $d_{h}(u,v)=d_{h}(u^{-},v^{-})=\ell .$ Let $%
p_{1}=(x_{0},x_{1},\cdots ,x_{\ell })$ be a shortest horizontal path from $%
u=x_{0}$ to $v=x_{\ell }$, and set $L(p_{1}^{-})=r,$ then $r\leq \ell .$ On
the other hand, $p_{1}^{-}$ is a horizontal walk joining $u^{-}$ and $v^{-},$
so $d_{h}(u^{-},v^{-})\leq r.$ Together with $d_{h}(u^{-},v^{-})=\ell $ this
yields $\ell \leq r\leq \ell ,$ and hence $r=\ell .$ Therefore, $%
x_{k}^{-}\neq x_{k+1}^{-}$ for all $0\leq k\leq \ell -1.$ By the definition
of horizontal edges, each $(x_{k},x_{k+1})$ is of $J$-type. Since no
consecutive parents were deleted, we obtain that 
\begin{equation*}
p:=p_{1}^{-}=(x_{0}^{-},x_{1}^{-},\cdots ,x_{\ell }^{-})
\end{equation*}%
and $p$ is a shortest horizontal path from $u^{-}$ to $v^{-}$. Let $%
\mathfrak{d}(p)=d_{1}d_{2}\cdots d_{\ell }$ be its direction sequence. Then
by the definition of the path matrix and (\ref{K_ij=1}), we have 
\begin{eqnarray*}
M(p)_{i,j} &=&(K_{d_{1}}K_{d_{2}}\cdots K_{d_{\ell }})_{i,j} \\
&\geq &(K_{d_{1}})_{\varphi (x_{0}),\varphi (x_{1})}(K_{d_{2}})_{\varphi
(x_{1}),\varphi (x_{2})}\cdots (K_{d_{\ell }})_{\varphi (x_{\ell
-1}),\varphi (x_{\ell })}=1>0.
\end{eqnarray*}

Conversely, suppose that $p=(z_{0},z_{1},\ldots ,z_{\ell })$ is a shortest
horizontal path from $u^{-}=z_{0}$ to $v^{-}=z_{\ell }$ and $M(p)_{i,j}>0$.
Let $\mathfrak{d}(p)=d_{1}d_{2}\cdots d_{\ell }$. Then there exist letters $%
c_{0},c_{1},\dots ,c_{\ell }\in \Sigma $ with $c_{0}=i$ and $c_{\ell }=j$
such that 
\begin{equation*}
(K_{d_{1}})_{c_{0},c_{1}}(K_{d_{2}})_{c_{1},c_{2}}\cdots (K_{d_{\ell
}})_{c_{\ell -1},c_{\ell }}=1.
\end{equation*}%
By the definition of $K_{0}=K$ and $K_{1}=K^{\mathrm{T}}$, each factor being 
$1$ means that at step $q$ there is a $J$-type horizontal edge in $\mathcal{U%
}$ between vertices $z_{q}c_{q},z_{q+1}c_{q+1}$, taken in the direction
prescribed by $d_{q}$. Therefore, for the path 
\begin{equation*}
(z_{0}c_{0},z_{1}c_{1},\dots ,z_{\ell }c_{\ell }),
\end{equation*}%
we obtain a horizontal path from $z_{0}c_{0}=u^{-}i=u$ to $z_{\ell }c_{\ell
}=v^{-}j=v$ of length $\ell $. Thus 
\begin{equation*}
d_{h}(u,v)\leq \ell =L(p)=d_{h}(u^{-},v^{-}),
\end{equation*}%
which together with (\ref{expansive}) implies that%
\begin{equation*}
d_{h}(u,v)=d_{h}(u^{-},v^{-}).
\end{equation*}%
This proves the converse implication and completes the proof.
\end{proof}

\begin{lemma}
\label{le-dhit}Let $x,y\in \Sigma _{n}$ be distinct vertices for some $n\in 
\mathbb{N}$ and $i,j\in \Sigma .$ If $d_{h}(xi,yj)=d_{h}(x,y),$ then 
\begin{equation}
d_{h}(xi^{t},yj^{t})=d_{h}(x,y)  \label{dhit}
\end{equation}
for all $t\geq 1,$ where $\sigma ^{t}$ denotes the concatenation of $t$
copies of the letter $\sigma $.
\end{lemma}

\begin{proof}
We argue by induction on $t$. The case $t=1$ is exactly the assumption.

Assume that (\ref{dhit}) holds for all $1\leq s\leq t$. Set 
\begin{equation*}
X_{s}:=xi^{s},\text{ }Y_{s}:=yj^{s}\text{ }(s\geq 1),
\end{equation*}%
then 
\begin{equation*}
d_{h}(X_{t},Y_{t})=d_{h}(X_{t-1},Y_{t-1}).
\end{equation*}%
Applying Lemma\ \ref{le-M} to the pair $(X_{t},Y_{t})$ (at level $n+t$) with
last letters $i=\varphi (X_{t})$ and $j=\varphi (Y_{t})$, we obtain a
shortest horizontal path 
\begin{equation*}
p_{t-1}=(z_{0},z_{1},\dots ,z_{\ell })
\end{equation*}%
from $X_{t-1}=z_{0}$ to $Y_{t-1}=z_{\ell }$ at level $n+t-1$ such that $%
M(p_{t-1})_{i,j}>0$ and $\ell =d_{h}(X_{t},Y_{t})=d_{h}(x,y)$. Furthermore,
as in the proof of Lemma \ref{le-M} (the "$\Leftarrow $" direction), we can
left one level up to obtain a shortest horizontal path $p_{t}$ whose
direction sequence coincides with that of $p_{t-1}.$ In particular, the
corresponding path matrix satisfies 
\begin{equation*}
M(p_{t})_{i,j}=M(p_{t-1})_{i,j}>0,
\end{equation*}%
and hence using Lemma \ref{le-M} again we obtain 
\begin{equation*}
d_{h}(X_{t+1},Y_{t+1})=d_{h}(X_{t},Y_{t})=d_{h}(x,y).
\end{equation*}%
Thus, by induction, the equality (\ref{dhit}) holds for all $t\geq 1$.
\end{proof}

We now state the main characterization of hyperbolicity for substitution
graphs.

\begin{theorem}
\label{th-equ}The substitution graph $\mathcal{U}$ is hyperbolic if and only
if there exists a positive integer $k$ such that for any shortest horizontal
path $p$ with length greater than $k,$ the associated path matrix satisfies $%
M(p)=\boldsymbol{0}$.
\end{theorem}

\begin{proof}
If $\mathcal{U}\ $is hyperbolic, then by Theorem \ref{th-Kong}, it is $(m,k)$%
-departing for some positive integers $m,k$. Let $p$ be a shortest
horizontal path from $x$ to $y$ with length $k+1.$ Assume, for the sake of
contradiction, that there exists some entry $(i,j)$ such that $M(p)_{i,j}>0.$
Then Lemma \ref{le-M} implies $d_{h}(xi,yj)=d_{h}(x,y)$, and consequently
Lemma \ref{le-dhit} yields $d_{h}(xi^{t},yj^{t})=d_{h}(x,y)\leq 2k$ for all $%
t\geq 1$. However, this contradicts the $(m,k)$-departing property, and
hence the path matrix of any shortest horizontal path of length greater than 
$k$ is the zero matrix.

On the other hand, if there exists a positive integer $k$ such that $M(p)=%
\boldsymbol{0}$\ for all shortest horizontal paths $p$ with $L(p)>k$. We
will show that $\mathcal{U}$ is $(k,k)$-departing.

Let $p$ be a shortest horizontal path from $x$ to $y$ and $L(p)>k$. Since $%
M(p)=\boldsymbol{0}$, by Lemma \ref{le-M} and the expansive property (\ref%
{expansive}) we have 
\begin{equation*}
d_{h}(xi_{1},yj_{1})\geq d_{h}(x,y)+1>k+1\text{ for all }i_{1},j_{1}\in
\Sigma .
\end{equation*}%
Thus, for every pair $i_{1},j_{1}$, any shortest horizontal path $p_{1}$
connecting $xi_{1}$ and $yj_{1}$ has length greater than $k$, so by the
hypothesis $M(p_{1})=\boldsymbol{0}$. Repeating this argument $k$ times, we
conclude that%
\begin{equation*}
d_{h}(xi_{1}i_{2}\cdots i_{k},yj_{1}j_{2}\cdots j_{k})>k+k=2k\text{ for all }%
i_{1}i_{2}\cdots i_{k},j_{1}j_{2}\cdots j_{k}\in \Sigma ^{k}.
\end{equation*}%
Therefore, $\mathcal{U}$ is $(k,k)$-departing, and hence hyperbolic.
\end{proof}

Based on this characterization, we now derive a combinatorial criterion for
the vanishing of path matrices, which offers a practical method to verify
the hyperbolicity condition in Theorem \ref{th-equ}.

Let $p$ be a horizontal path in $\mathcal{U},$ then $M(p)\in \left\langle
K,K^{\mathrm{T}}\right\rangle .$ Recall the notations $N_{0}$ and $N_{1}$
defined for matrices. For simplicity, we extend these to paths by setting%
\begin{equation*}
N_{0}(p):=N_{0}(M(p))\text{ and}\,N_{1}(p):=N_{1}(M(p)).
\end{equation*}

\begin{proposition}
\label{prop-Sp}Suppose $K$ is a nilpotent matrix with nilpotency index $I$
and its associated poset $(P,\leq )$ is graded. Let $p$ be a horizontal path
in $\mathcal{U}.$ Then $M(p)=\boldsymbol{0}$ if and only if there exists a
subpath $p^{\prime }$ of $p$ such that $\left\vert N_{0}(p^{\prime
})-N_{1}(p^{\prime })\right\vert \geq I.$
\end{proposition}

\begin{proof}
Let 
\begin{equation*}
\mathcal{M}=\{M(p):p\text{ is a shortest horizontal path in }\mathcal{U}\}.
\end{equation*}%
Since $\mathcal{M}\subset \left\langle K,K^{\mathrm{T}}\right\rangle $,
applying Lemma \ref{le-SN} completes the proof.
\end{proof}

Combining Theorem \ref{th-equ} and Proposition \ref{prop-Sp}, we obtain the
following.

\begin{corollary}
Suppose $K$ is a nilpotent matrix with nilpotency index $I$ and its
associated poset $(P,\leq )$ is graded. Then $\mathcal{U}$ is hyperbolic if
and only if there exists a positive integer $k$ such that every shortest
horizontal path $p$ of length greater than $k$ contains a subpath $p^{\prime
} $ satisfying $\left\vert N_{0}(p^{\prime })-N_{1}(p^{\prime })\right\vert
\geq I$.
\end{corollary}

We define $S(M)$ as the length of the longest maximal contiguous block of
equal factors (all $A$'s or all $A^{\mathrm{T}}$'s) in $M\in \left\langle
K,K^{\mathrm{T}}\right\rangle $ and write $S(p):=S(M(p))$. By the definition
of the nilpotent matrix and Theorem \ref{th-equ}, we have the following
corollary.

\begin{corollary}
\label{coro-N0-N1}Suppose $K\ $is a nilpotent matrix with nilpotency index $%
I $. If $S(p)\geq I,$ then $M(p)=\boldsymbol{0}.$ Consequently, If $S(p)\geq
I$ for all shortest horizontal paths $p$ of length greater than $k,$ then $%
\mathcal{U}$ is hyperbolic.
\end{corollary}

Theorem \ref{th-equ} characterizes hyperbolicity but involves verifying
infinitely many paths. We now deduce finite, verifiable criteria from this
general result in the following two subsections.

\bigskip

\subsection{Obstructions to hyperbolicity}

For each $n\in \mathbb{N},$ let $D_{n}$ denote the horizontal diameter at
level $n$ in $\mathcal{U},$ defined by%
\begin{equation*}
D_{n}=\max \{d_{h}(x,y):x,y\in \Sigma _{n}\}.
\end{equation*}

\begin{proposition}
\label{prop-nil}If $\mathcal{U}$ is hyperbolic, then either $\{D_{n}\}_{n\in 
\mathbb{N}}$ is bounded or $K$ is nilpotent.
\end{proposition}

\begin{proof}
Assume that $\{D_{n}\}_{n\in \mathbb{N}}$ is unbounded and $K$ is not
nilpotent. Since the diameter sequence is unbounded, for any integer $k\geq
1,$ there exists a shortest horizontal path $p$ with length strictly greater
than $k$. By Lemma \ref{le-A1Ak}, the non-nilpotency of $K$ implies that $%
M(p)\neq \boldsymbol{0}$. This contradicts the hyperbolicity condition in
Theorem \ref{th-equ}.
\end{proof}

A shortest horizontal path $p$ in $\mathcal{U}$ with direction sequence $%
\mathfrak{d}(p)=d_{1}d_{2}\cdots d_{\ell }$ is said to be \emph{strictly
alternating} if $d_{i}\neq d_{i+1}$ for all $1\leq i<\ell $. This leads to
the following obstruction.

\begin{proposition}
\label{prop-arbi}If there exist arbitrarily long strictly alternating
arbitrarily long in $\mathcal{U}$, then $\mathcal{U}$ is not hyperbolic.
\end{proposition}

\begin{proof}
Since there exist arbitrarily long shortest horizontal paths, the diameter
sequence $\{D_{n}\}_{n\in \mathbb{N}}$ is unbounded.\newline
If $K$\ is not nilpotent. By Proposition \ref{prop-nil}, the unboundedness
of $\{D_{n}\}$ together with the non-nilpotency of $K$ implies that $%
\mathcal{U}$ is not hyperbolic.\newline
If $K$\ is nilpotent, then $K\neq 0$. For any strictly alternating path $p$,
the matrix $M(p)$ is a strictly alternating product of $K$ and $K^{T}$. By
Remark \ref{rmk0}, such products never vanish when $K\neq 0$. Hence there
exist shortest horizontal paths of arbitrarily large length whose path
matrices are non-zero. Therefore, by Theorem \ref{th-equ}, that $\mathcal{U}$
is not hyperbolic.
\end{proof}

Recall that $\overrightarrow{G}$ is the directed graph on the alphabet $%
\Sigma ,$ and $\overrightarrow{J}$ is the subgraph of $\overrightarrow{K}%
_{N,N}$ (connecting $\Sigma $ to a disjoint copy $\Sigma ^{\prime }$).

\begin{definition}
\label{def-coup}We define the \emph{coupling graph} $\overrightarrow{%
\mathcal{G}}\mathcal{=}\overrightarrow{\mathcal{G}}(\overrightarrow{G},%
\overrightarrow{J})$ as the directed graph with vertex set $V(%
\overrightarrow{\mathcal{G}})=\Sigma \cup \Sigma ^{\prime }$ and edge set 
\begin{equation*}
E(\overrightarrow{\mathcal{G}})=E(\overrightarrow{G})\cup E(\overrightarrow{%
G^{\prime }})\cup E(\overrightarrow{J}),
\end{equation*}%
where $\overrightarrow{G^{\prime }}$ is the isomorphic copy of $%
\overrightarrow{G}$ on $\Sigma ^{\prime }$ (i.e., $(u^{\prime },v^{\prime
})\in E(\overrightarrow{G^{\prime }})$ if and only if $(u,v)\in E(%
\overrightarrow{G})$). We denote by $\mathcal{G}$ the underlying graph of $%
\overrightarrow{\mathcal{G}}$.
\end{definition}

Fig. \ref{fig-coupling} illustrates the coupling graph corresponding to
Example~1, where $\overrightarrow{G}$ is a directed path of length $2$
(i.e., $1\rightarrow 2\rightarrow 3$) and $E(\overrightarrow{J}%
)=\{(2,1^{\prime })\}$.

\begin{figure}[tbph]
\includegraphics[width=0.4\textwidth]{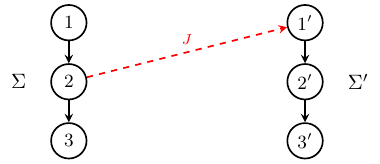}
\caption{Coupling graph for Example~1.}
\label{fig-coupling}
\end{figure}

For each $n\in \mathbb{N},$ let $\overrightarrow{G}_{n}$ denote the subgraph
of the directed substitution graph $\mathcal{D}=(V,\vec{E}_{h})$ induced by
the vertex subset $\Sigma _{n}.$ We write $G_{n}$ for the underlying graph
of $\overrightarrow{G}_{n}$. Then the set of horizontal edges can be
expressed as 
\begin{equation*}
\vec{E}_{h}=\bigcup_{n=1}^{\infty }E(\overrightarrow{G}_{n})\text{ and }%
E_{h}=\bigcup_{n=1}^{\infty }E(G_{n}).
\end{equation*}

\begin{remark}
The recursive construction of the horizontal graphs $G_{n}$ preserves
certain elementary graph--theoretic properties of the base graphs on the
alphabet level. For instance, if the underlying graph $\mathcal{G}$ is
triangle-free, then each horizontal graph $G_{n}$ is triangle-free for all $%
n\in \mathbb{N}$. This can be proved by a straightforward induction on $n$
using the definition of substitution graph.

Motivated by this observation, we would also like to propagate bipartiteness
from the alphabet level to all horizontal layers. For this we fix a
bipartition of $\mathcal{G}$ and impose a simple compatibility requirement
between this bipartition and the substitution rules, which we formulate as
the notion of a \emph{shift-compatible} bipartition below. Under this
condition, we will prove in Lemma~\ref{le-bipart} that each $G_{n}$ is
bipartite.
\end{remark}

\begin{definition}
Suppose $\mathcal{G}$ is bipartite. Let $\Lambda $ be a undirected graph
defined by $V(\Lambda )=\Sigma $ and%
\begin{equation}
(i,j)\in E(\Lambda )\iff \text{either }(i,j)\in E(G)\text{ or }(i,j^{\prime
})\in E(J).  \label{Lambda}
\end{equation}%
We say that a proper $2$-coloring $\chi $ of $\mathcal{G}$ is called \emph{%
shift-compatible} if the parity shift 
\begin{equation*}
\tau (i):=\chi (i^{\prime })-\chi (i)\pmod2\text{ }(i\in \Sigma )
\end{equation*}%
is constant on each connected component of $\Lambda $. We call $\mathcal{G}$
a \emph{shift-compatible bipartite graph }if it admits such a $2$-coloring $%
\chi $.
\end{definition}

\begin{remark}
\label{rmk-equi}The above definition admits a convenient local
reformulation. Suppose $\mathcal{G}$ is bipartite. A proper $2$-coloring $%
\chi $ of $\mathcal{G}$ is shift-compatible if and only if 
\begin{equation}
\tau (i)=\tau (j)\text{ for every }(i,j)\in E(\Lambda ).  \label{shift-c}
\end{equation}%
In other words, $\tau $ is constant on each edge of $\Lambda $. Indeed, if $%
\tau $ is constant on each connected component of $\Lambda $, then in
particular $\tau (i)=\tau (j)$ for every edge $\{i,j\}\in E(\Lambda )$, so
the local condition holds. Conversely, if $\tau $ is constant on every edge
of $\Lambda $, then for any two vertices $i,j$ in the same connected
component of $\Lambda $ we can choose a path $i=i_{0},i_{1},\dots ,i_{k}=j$
in $\Lambda $ and apply the edge condition successively to obtain $\tau
(i_{0})=\tau (i_{1})=\cdots =\tau (i_{k})$. Thus $\tau $ is constant on that
component, and the global definition and the local characterization are
equivalent.
\end{remark}

\begin{lemma}
\label{le-bipart}If $\mathcal{G}$ is a shift-compatible bipartite graph,
then $G_{n}$ is bipartite for each $n\in \mathbb{N}$.
\end{lemma}

\begin{proof}
We prove by induction on $n$. For base case $n=1,$ by definition, $G_{1}$ is
isomorphic to $G$. Since $G\subset \mathcal{G}$ and $\mathcal{G}$ is
bipartite, $G_{1}$ is bipartite.

Assume $G_{n-1}$ is bipartite and fix a $2$-coloring $C_{n-1}:V(G_{n-1})%
\rightarrow \{0,1\}$. Let $\chi :V(\mathcal{G})\rightarrow \{0,1\}$ be a
shift-compatible proper $2$-coloring of $\mathcal{G}$. Define 
\begin{equation*}
C_{n}(ui)\equiv \chi (i)+\tau (i)C_{n-1}(u)\pmod2,\text{ }u\in \Sigma
_{n-1},i\in \Sigma .
\end{equation*}%
We show that $C_{n}$ is a proper $2$-coloring of $G_{n}$.

Let $(ui,uj)$ be a $G$-type horizontal edge in $G_{n}$. Then $(i,j)\in E(G)$
and (\ref{Lambda})--(\ref{shift-c}) yield $\tau (i)=\tau (j)$, and hence 
\begin{eqnarray*}
C_{n}(ui)+C_{n}(uj) &\equiv &\chi (i)+\tau (i)C_{n-1}(u)+\chi (j)+\tau
(j)C_{n-1}(u)\pmod2 \\
&\equiv &\chi (i)+\chi (j)\pmod2.
\end{eqnarray*}%
Since $(i,j)\in E(G)\subset E(\mathcal{G})$ and $\chi $ is a proper $2$%
-coloring of $\mathcal{G}$, we have $\chi (i)\neq \chi (j)$. Thus $%
C_{n}(ui)\neq C_{n}(uj)$.

Let $(ui,uj)$ be a $J$-type horizontal edge in $G_{n}$. Then $(u,v)\in
E(G_{n-1})$ and $(i,j^{\prime })\in E(J)\subset E(\mathcal{G})$. Since $\chi 
$ is a proper $2$-coloring, it follows that%
\begin{equation}
\chi (i)+\chi (j^{\prime })\equiv 1\pmod2.  \label{hang}
\end{equation}%
Using the definition of the parity shift $\tau \ $and (\ref{Lambda})--(\ref%
{hang}) (such that $\tau (i)=\tau (j)=:\tau $), we obtain 
\begin{equation}
\chi (i)+\chi (j)\equiv 1-\tau \pmod2.  \label{uni}
\end{equation}%
Since $C_{n-1}$ is a proper $2$-coloring of $G_{n-1}$, we have%
\begin{equation}
C_{n-1}(u)+C_{n-1}(v)\equiv 1\pmod2.  \label{bei}
\end{equation}%
It follows from (\ref{uni})--(\ref{bei}) that%
\begin{align*}
C_{n}(ui)+C_{n}(vj)& \equiv \left( \chi (i)+\tau C_{n-1}(u)\right) +\left(
\chi (j)+\tau C_{n-1}(v)\right) \pmod2 \\
& \equiv \chi (i)+\chi (j)+\tau \left( C_{n-1}(u)+C_{n-1}(v)\right) \pmod2 \\
& \equiv (1-\tau )+\tau \cdot 1\equiv 1\pmod2.
\end{align*}%
Hence $C_{n}(ui)\neq C_{n}(vj)$.

In both cases, adjacent vertices in $G_{n}$ receive different colors under $%
C_{n}$. Therefore, $C_{n}$ is a proper $2$-coloring and $G_{n}$ is
bipartite. This completes the induction.
\end{proof}

Let $\deg _{\Gamma }^{+}(u)$ and $\deg _{\Gamma }^{-}(u)$ denote the
out-degree and in-degree of a vertex $u$ in a directed graph $\Gamma $,
respectively. For substitution graph, we specifically define the horizontal
degrees as 
\begin{equation*}
\deg _{h}^{+}(u)=\sharp \{(u,v):(u,v)\in \vec{E}_{h}\}\text{ and }\deg
_{h}^{-}(u)=\sharp \{(v,u):(v,u)\in \vec{E}_{h}\}.
\end{equation*}

\begin{proposition}
\label{prop-bi}Suppose $\mathcal{G}\ $is a shift-compatible bipartite graph.
If there exists an edge $(i,j^{\prime })\in E(\overrightarrow{J})$ such that%
\begin{equation*}
\deg _{\overrightarrow{G}}^{+}(i)\geq 1\text{ and }\deg _{\overrightarrow{G}%
}^{-}(j)\geq 1,
\end{equation*}%
then $\mathcal{U}$ is not hyperbolic.
\end{proposition}

\begin{proof}
Let $(i,j^{\prime })$ be an edge in $\overrightarrow{J}$ satisfying the
given degree conditions. Since $\deg _{\overrightarrow{G}}^{+}(i)\geq 1$ and 
$\deg _{\overrightarrow{G}}^{-}(j)\geq 1,$ there exists $i_{\mathrm{out}},j_{%
\mathrm{in}}\in \Sigma $\ such that $(i,i_{\mathrm{out}}),(j_{\mathrm{in}%
},j)\in E(\overrightarrow{G})$. Consider an arbitrary edge $(u,v)\in E(%
\overrightarrow{G}_{1})$. At the level $2$, consider the vertices $ui,vj,vj_{%
\mathrm{in}}.$ Since $(i,j^{\prime })\in E(\overrightarrow{J})$ and $(j_{%
\mathrm{in}},j)\in E(\overrightarrow{G}),$ it follows that $(ui,vj),(vj_{%
\mathrm{in}},vj)\in \vec{E}_{h}.$ Thus, $p_{2}$ forms an alternating
horizontal path of length $2$ with direction sequence $01$. By Lemma \ref%
{le-bipart}, $\overrightarrow{G}_{2}$ is bipartite, hence $p_{2}$ is a
alternating shortest horizontal path. We define $p_{n}$ recursively as
follows. For $n>2$, write 
\begin{equation*}
p_{n-1}=(v_{0},v_{1},\dots ,v_{n-1})
\end{equation*}%
for the path at level $n-1$, and define $p_{n}$ at level $n$ by%
\begin{equation*}
p_{n}=\left\{ 
\begin{array}{cc}
(v_{0}i,v_{1}j,\dots ,v_{n-2}i,v_{n-1}j,v_{n-1}j_{\mathrm{in}}), & \text{if }%
n\text{ is even,} \\ 
(v_{0}i,v_{1}j,\dots ,v_{n-2}j,v_{n-1}i,v_{n-1}i_{\mathrm{out}}), & \text{if 
}n\text{ is odd,}%
\end{array}%
\right.
\end{equation*}%
where $v_{k}\in \Sigma _{n-1}$, and $v_{k}i$, $v_{k}j$, $v_{n-1}j_{\mathrm{in%
}}$, $v_{n-1}i_{\mathrm{out}}$ are their children in $\Sigma _{n}$. Then $%
p_{n}$ has direction sequence 
\begin{equation*}
\mathfrak{d(}p_{n})=\left\{ 
\begin{array}{cc}
\mathfrak{d(}p_{n-1})\mathfrak{d(}v_{n-1}j,v_{n-1}j^{\prime })=0101\cdots 01,
& \text{if }n\text{ is even,} \\ 
\mathfrak{d(}p_{n-1})\mathfrak{d}(v_{n-1}i,v_{n-1}i^{\prime })=0101\cdots
010, & \text{if }n\text{ is odd,}%
\end{array}%
\right.
\end{equation*}%
which implies that $p_{n}$ is an alternating horizontal path at level $n$.

We now prove by induction that $p_{n}$ is a shortest horizontal path for all 
$n\geq 2$.

Suppose $p_{n-1}$ is a\ shortest horizontal path of length $n-1$. Let $q$ be
any shortest horizontal path connecting the endpoints $w_{0}$ and $w_{n}$ of 
$p_{n},$ with length $\ell =d_{h}(w_{0},w_{n}).$ By Lemma \ref{le-bipart}, $%
\overrightarrow{G}_{n}$ is bipartite, so any two vertices on $%
\overrightarrow{G}_{n}$ belong either to the same part or to different
parts, and all paths joining them have lengths of the same parity. As $p_{n}$
has length $n,$ it follows that every horizontal path connecting its
endpoints has length congruent to $n$ modulo $2,$ i.e. 
\begin{equation}
\ell \equiv n\pmod 2.  \label{parity}
\end{equation}%
Suppose, for the sake of contradiction, that $\ell <n$. Then the parity
condition (\ref{parity}) implies $\ell \leq n-2.$ However, by the expansive
property (\ref{expansive}) we have%
\begin{equation*}
\ell =d_{h}(v_{0},v_{n})\geq d_{h}(v_{0}^{-},v_{n}^{-})=L(p_{n-1})=n-1,
\end{equation*}%
which together with $\ell \leq n-2$ leads to the contradiction $n-1\leq \ell
\leq n-2.$ Hence we have $\ell \geq n,$ which implies that $p_{n}$ is a
shortest horizontal path. The existence of arbitrarily long alternating
shortest paths implies, by Proposition \ref{prop-arbi}, that $\mathcal{U}$
is not hyperbolic.
\end{proof}

\begin{lemma}
\label{le-conne_bi}If $G$ is connected and $\mathcal{G}$ is bipartite, then $%
\mathcal{G}$ admits a shift-compatible bipartition.
\end{lemma}

\begin{proof}
Since $\mathcal{G}$ is bipartite, there exists a proper $2$-coloring $\chi
:V(\mathcal{G})\rightarrow \{0,1\}.$ Let $(i,j)\in E(G)$, then also $%
(i^{\prime },j^{\prime })\in E(G^{\prime })\subset E(\mathcal{G}).$ Since $%
\chi $ is a proper $2$-coloring, we have 
\begin{equation*}
\chi (i)+\chi (j)\equiv 1\pmod 2,\text{ }\chi (i^{\prime })+\chi (j^{\prime
})\equiv 1\pmod 2.
\end{equation*}%
Subtracting gives 
\begin{equation*}
\tau (i)+\tau (j)=(\chi (i^{\prime })-\chi (i))+(\chi (j^{\prime })-\chi
(j))\equiv 0\pmod 2,
\end{equation*}%
hence 
\begin{equation}
\tau (i)=\tau (j).  \label{adj}
\end{equation}

Since $G$ is connected, any two letters $u,v\in \Sigma $ can be joined by a
path $u=u_{0},\ldots ,u_{k}=v$ in $G$. Iterating the above equality (\ref%
{adj}) along this path yields%
\begin{equation*}
\tau (u)=\tau (u_{1})=\cdots =\tau (u_{k})=\tau (v),
\end{equation*}%
hence $\tau (u)=\tau (v)$ for all $u,v\in \Sigma .$ Thus $\tau $ is constant
on $\Sigma ,$ and hence $\chi $ is shift-compatibleby definition.
\end{proof}

Combining Proposition \ref{prop-bi} and Lemma \ref{le-conne_bi}, we obtain
the following.

\begin{corollary}
\label{coro-nonhyper}Suppose $G$ is connected and $\mathcal{G}$ is
bipartite. If there exists an edge $(i,j^{\prime })\in E(\overrightarrow{J})$
such that%
\begin{equation*}
\deg _{\overrightarrow{G}}^{+}(i)\geq 1\text{ and }\deg _{\overrightarrow{G}%
}^{-}(j)\geq 1,
\end{equation*}%
then $\mathcal{U}$ is not hyperbolic.
\end{corollary}

One can observe that if $G$ is bipartite, and $\sharp E(J)=1,$ then $%
\mathcal{G}$ is also bipartite. Since every vertex in a directed cycle has
both in-degree and out-degree equal to $1$, by Corollary \ref{coro-nonhyper}
we immediately have the following.

\begin{corollary}
\label{coro-cycle}Suppose $G$ is connected and bipartite. If $%
\overrightarrow{J}$ consists of a single edge $(i,j^{\prime })$ such that
both $i$ and $j$ lie on directed cycles of $\overrightarrow{G},$ then $%
\mathcal{U}$ is not hyperbolic.
\end{corollary}

\bigskip

\subsection{Structural criteria for hyperbolicity}

In this subsection we provide two structural sufficient conditions for
substitution graph $\mathcal{U}$ to be hyperbolic. One in terms of simple
combinatorial constraints on the coupling graph $\overrightarrow{\mathcal{G}}
$, and another formulated via a local arborescence condition on $%
\overrightarrow{G}$ and the acyclicity of $\mathcal{G}.$

\begin{lemma}
\label{le-<=1}The following properties hold:

\begin{enumerate}
\item[(i)] If every vertex in $\overrightarrow{\mathcal{G}}$ has out--degree
at most $1$, then every vertex in $\mathcal{D}$ has horizontal out--degree
at most $1$.

\item[(ii)] If every vertex in $\overrightarrow{\mathcal{G}}$ has in--degree
at most $1$, then every vertex in $\mathcal{D}$ has horizontal in--degree at
most $1$.
\end{enumerate}
\end{lemma}

\begin{proof}
It suffices to prove (i), as the proof of (ii) is analogous. Assume, for
contradiction, that there exists a vertex $x\in \Sigma _{n}$ (for some $n$)
with horizontal out--degree at least two in level $n,$ i.e. 
\begin{equation*}
\deg _{\overrightarrow{G}_{n}}^{+}(x)\geq 2,
\end{equation*}%
and choose such an $n$ minimal. Then there are two distinct vertices $y,z\in
\Sigma _{n}$ such that $(x,y)\in \vec{E}_{h}$ and $(x,z)\in \vec{E}_{h}$.
Write 
\begin{equation*}
x=x^{-}a,y=y^{-}b,z=z^{-}c.
\end{equation*}%
Set 
\begin{equation*}
\eta (x,y):=\left\{ 
\begin{array}{cc}
(a,b), & \text{if }(x,y)\text{ is of }G\text{-type,} \\ 
(a,b^{\prime }), & \text{if }(x,y)\text{ is of }J\text{-type.}%
\end{array}%
\right.
\end{equation*}%
Similarly we define $\eta (x,z).$ By construction, both $\eta (x,y)$ and $%
\eta (x,z)$\ are edges of $\overrightarrow{\mathcal{G}}$ with source $a$.
Since $\deg _{\overrightarrow{\mathcal{G}}}^{+}(a)\leq 1,$ all outgoing
edges from $a$ in $\overrightarrow{\mathcal{G}}$ must coincide, hence 
\begin{equation*}
\eta (x,y)=\eta (x,z).
\end{equation*}%
Consequently, $(x,y)$ and $(x,z)$ are of the same type, and hence $b=c.$

If both $(x,y)$ and $(x,z)$ are of $G$-type, then $x^{-}=y^{-}=z.$ It
follows from $b=c$ that 
\begin{equation*}
y=y^{-}b=z^{-}c=z
\end{equation*}%
contradicting $y\neq z.$

If both $(x,y)$ and $(x,z)$ are of $J$-type, then $%
(x^{-},y^{-}),(x^{-},z^{-})\in E(\overrightarrow{G}_{n-1}).$ Since $y\neq z,$
we have $y^{-}\neq z^{-},$ and hence 
\begin{equation*}
\deg _{\overrightarrow{G}_{n-1}}^{+}(x^{-})\geq 2,
\end{equation*}%
which contradicts the minimality of $n$.

In above both cases we obtain a contradiction. Hence no such vertex $x$ can
exist, and $\deg _{\overrightarrow{G}_{n}}^{+}(x)\leq 1$ for all $x\in
\Sigma _{n}$ and $n\geq 1,$ i.e. every vertex in $\mathcal{D}$ has
horizontal out--degree at most $1$.
\end{proof}

\begin{theorem}
\label{th-degree}The substitution graph $\mathcal{U}$ is hyperbolic provided
that

\begin{enumerate}
\item[(i)] $K$ is nilpotent, and

\item[(ii)] either all vertices in $\overrightarrow{\mathcal{G}}$ have
out-degree at most 1, or all vertices have in-degree at most 1.
\end{enumerate}
\end{theorem}

\begin{proof}
We only prove the case when all vertices in $\overrightarrow{\mathcal{G}}$
have out-degree at most $1;$ the case of bounded in-degree is analogous.

Let $I$ be the nilpotency index of the matrix $K(J)$. Consider a shortest
horizontal path $p$ in $\mathcal{U}$ with length $L(p)=2I-1$ and denote its
direction sequence by%
\begin{equation*}
\mathfrak{d(}p)=d_{1}d_{2}\cdots d_{2I-1}\in \{0,1\}^{2I-1}.
\end{equation*}

By assumption (ii) and Lemma~\ref{le-<=1}(i), every vertex in $\mathcal{D}$
has horizontal out--degree at most $1$. In particular, along any shortest
horizontal path the direction cannot change from $1$ to $0$. Indeed, if
there exists an index $i\in \{1,\dots ,2I-2\}$ such that $d_{i}=1$ and $%
d_{i+1}=0$, then at the intermediate vertex $v$ the path would use two
distinct outgoing horizontal edges of different types, which contradicts $%
\deg _{h}^{+}(v)\leq 1$.

Since the direction sequence $\mathfrak{d(}p)$ has length $2I-1$ and
contains no subsequence $10$, there exists an index $j$ such that 
\begin{equation*}
d_{j}=d_{j+1}=\cdots =d_{j+I-1},
\end{equation*}%
and hence $S(p)\geq I.$ This implies that $S(p)\geq I$ for all shortest
horizontal paths $p$ of length greater than $k=2I-2.$ Thus $\mathcal{U}$ is
hyperbolic by Corollary \ref{coro-N0-N1}.
\end{proof}

An \emph{arborescence} is a directed graph with a designated root vertex $%
\vartheta $ such that for every other vertex $u$, there exists exactly one
directed path from $\vartheta $ to $u$. In other words, an arborescence is a
directed, rooted tree in which all edges point away from the root. An \emph{%
anti-arborescence} is defined dually as a directed rooted tree in which all
edges point towards the root $\vartheta $, or equivalently, for every vertex 
$u\neq \vartheta $ there exists exactly one directed path from $u$ to $%
\vartheta .$

Note that every vertex in an arborescence has in-degree at most $1$, while
every vertex in an anti-arborescence has out-degree at most $1$. Thus we
obtain the following corollary by Theorem \ref{th-degree}.

\begin{corollary}
\label{cor-arbo}If $K$ is nilpotent and either all connected components of $%
\mathcal{G}$ are arborescences or all connected components of $\mathcal{G}$
are anti-arborescences, then $\mathcal{U}$ is hyperbolic.
\end{corollary}

The structural requirement in Corollary \ref{cor-arbo} applies to the entire
connected component. To obtain a more flexible criterion, we now relax this
global constraint and introduce a local version of the arborescence
structure.

Let $\Gamma $ be a directed graph and $F\subset V(\Gamma ).$ We denote by $%
\Gamma \lbrack F]$ the subgraph of $\Gamma $ induced by $F$.

\begin{definition}
\label{LAC}Let $\Gamma $ be a directed graph, $S\subset V(\Gamma )$ and $%
t\in \mathbb{N}$. We say that $\Gamma $ satisfies the $(S,t)$\emph{-local
arborescence condition (}$(S,t)$-LAC\emph{)} if for every $x\in S$, the
subgraph $\Gamma \lbrack U_{t}(x)]$ of $\Gamma $ induced by the $t$%
-neighborhood $U_{t}(x)$ of $x$ satisfies the following properties:

\begin{enumerate}
\item[(i)] $\Gamma \lbrack U_{t}(x)]$ is either an arborescence or an
anti-arborescence (rooted at $x$), and

\item[(ii)] either the height of $\Gamma \lbrack U_{t}(x)]$ equals $t$, or $%
\Gamma \lbrack U_{t}(x)]$ is an isolated connected component of $\Gamma $
that intersects $S$ only at $x$.
\end{enumerate}

Furthermore, we say that $\Gamma $ satisfies the $(S_{1},S_{2},t)$\emph{%
-local arborescence condition }($(S_{1},S_{2},t)$-LAC) if it satisfies the
following properties:

\begin{enumerate}
\item[(i)] $\Gamma $ satisfies the $(S_{1}\cup S_{2},t)$-LAC,

\item[(ii)] For every $x\in S_{1}$, $\Gamma \lbrack U_{t}(x)]$ is an
arborescence, and

\item[(iii)] For every $x\in S_{2}$, $\Gamma \lbrack U_{t}(x)]$ is an
anti-arborescence.
\end{enumerate}
\end{definition}

\begin{theorem}
\label{th-lac}Suppose $K$ is nilpotent and the nilpotency index of $K$ is $I$%
. Let 
\begin{equation*}
S_{1}=\{j\in \Sigma :\text{the }j\text{-th column of }K\text{ is not the
zero vector}\}
\end{equation*}%
and 
\begin{equation*}
S_{2}=\{i\in \Sigma :\text{the }i\text{-th row of }K\text{ is not the zero
vector}\}.
\end{equation*}%
If $\overrightarrow{G}$ satisfies the $(S_{1},S_{2},I-1)$-LAC and $\mathcal{G%
}$ is acyclic, then $\mathcal{U}$ is hyperbolic.
\end{theorem}

\begin{remark}
\label{rmk3}By definition, $S_{1}$ (resp.\ $S_{2}$) consists of those
letters whose corresponding column (resp.\ row) of $K$ is non-zero. If $%
(u,v)\in \overrightarrow{E}_{h}$ is of $J$-type, then the entry of $K$
indexed by $(\varphi (u),\varphi (v))$ is non-zero, so $\varphi (u)\in S_{2}$
and $\varphi (v)\in S_{1}$; in particular, 
\begin{equation}
\varphi (u),\varphi (v)\in S_{1}\cup S_{2}.  \label{S1S2}
\end{equation}%
Moreover, $(S_{1},S_{2},t)$-LAC implies that if $x\in S_{1}\cap S_{2}$, then 
$\Gamma \lbrack U_{t}(x)]$ must form an isolated connected component
(otherwise it could not be both an arborescence and an anti-arborescence).
\end{remark}

\begin{claim}
\label{clm0}Suppose $\overrightarrow{G}$ satisfies the $(S_{1},S_{2},I-1)$%
-LAC and $(u,v),(v,w)\in E_{h}$. If $(u,v)$ is of $J$-type and $(v,w)$ is of 
$G$-type, then the two edges have the same direction$,$ i.e. 
\begin{equation*}
\mathfrak{d}(u,v)=\mathfrak{d}(v,w).
\end{equation*}
\end{claim}

\begin{proof}
Without loss of generality we can assume that $(u,v)\in \overrightarrow{E}%
_{h}.$ Then by the definition of $J$-type horizontal edge we have $%
K_{\varphi (u),\varphi (v)}=1,$ and hence $\varphi (v)\in S_{1}.$ Since $%
\overrightarrow{G}$ satisfies the $(S_{1},S_{2},I-1)$-LAC, $%
H:=G[U_{t}(\varphi (v))]$ is a arborescence rooted at $\varphi (v),$ which
together with the definition of $G$-type horizontal edge imply that 
\begin{equation*}
(\varphi (v),\varphi (w))\in E(H).
\end{equation*}%
Then we obtain that 
\begin{equation*}
\mathfrak{d}(u,v)=0=\mathfrak{d}(v,w).
\end{equation*}
\end{proof}

Let $p=(x_{0},x_{1},\dots ,x_{t})$ be a horizontal path in $\mathcal{U}$ and
let $p^{-}=(y_{0},y_{1},\dots ,y_{\ell })$ be its projected walk. We first
consider the case $\ell \in \{0,1\}$, i.e.\ $0\leq L(p^{-})\leq 1$. In this
case we define a map%
\begin{equation*}
\Phi _{p}:\{x_{0},\dots ,x_{t}\}\longrightarrow \Sigma \cup \Sigma ^{\prime }
\end{equation*}%
by 
\begin{equation*}
\Phi _{p}(x_{i})=%
\begin{cases}
\varphi (x_{i}), & \text{if }x_{i}^{-}=y_{0}, \\[0.3ex] 
\varphi (x_{i})^{\prime }, & \text{if }\ell =1\text{ and }x_{i}^{-}=y_{1}.%
\end{cases}%
\end{equation*}%
We then set 
\begin{equation*}
\Phi (p):=\left( \Phi _{p}(x_{0}),\Phi _{p}(x_{1}),\dots ,\Phi
_{p}(x_{t})\right) .
\end{equation*}%
By the definition of horizontal edges, the map $\Phi _{p}$ is a graph
homomorphism from the path $p$ to the coupling graph $\mathcal{G}$.
Consequently, $\Phi (p)$ constitutes a directed walk in $\mathcal{G}$.

To prove Theorem \ref{th-lac}, we need to ensure that geodesics do not
deviate unpredictably from the hierarchical structure. This motivates the
following geometric property. We say that a substitution graph $\mathcal{U}$
satisfies the \emph{geodesic projection property} (GPP) if for every
shortest horizontal path $p$ in $\mathcal{U}$, its projected walk $p^{-}$ is
also a shortest horizontal path.

\begin{lemma}
\label{le-GPP}If $\mathcal{G}$ is acyclic, then $\mathcal{U}$ satisfies the
GPP.
\end{lemma}

\begin{proof}
We first claim that $G_{n}$ is acyclic for all $n\geq 1.$ We proceed by
induction on $n$. For $n=1$, $G_{1}$ is isomorphic to the subgraph $G$ of $%
\mathcal{G}$, hence $G_{1}$ is acyclic. Assume that $G_{n-1}$ is acyclic.
Assume that $G_{n}$ contains a cycle $\mathcal{C}$. Then $\mathcal{C}^{-}$
must be trivial. Consequently, all vertices of $\mathcal{C}$ share the same
parent $u$, i.e. $L(\mathcal{C}^{-})=0$. Then $\Phi (\mathcal{C})$ is also a
cycle on $\mathcal{G}$, which contradicts the acyclicity of $\mathcal{G}$.
Therefore, $G_{n}$ is acyclic.

We now verify that $\mathcal{U}$ satisfies the GPP. Let $p$ be a shortest
horizontal path in $\mathcal{U}$ connecting vertices $u$ and $v$ at level $n$%
. We show that $p^{-}$ is a simple path; the uniqueness of shortest paths in
a forest then implies the GPP.

Suppose $p^{-}$ contains repeated vertices. Since $G_{n-1}$ is acyclic, any
non-simple walk between two vertices must backtrack. That is, $p^{-}$
contains a subsequence of the form $x,y,\dots ,y,x$. In the child level,
this would require $p$ to traverse from the $J_{1}(x)$ to $J_{1}(y)$ and
return to $J_{1}(x)$ (where $J_{1}(w)$ denote the $1$-th descendant set of $%
w\in V$). Hence $p$ contains a repeated edge, which contradicts the fact
that $p$ is path.
\end{proof}

Let 
\begin{equation*}
L_{0}:=\max \{L(q):M(q)\neq \mathbf{0},\text{ }q\text{ is a shortest path in 
}\mathcal{G}\}.
\end{equation*}%
Let $p=(x_{0},x_{1},\dots ,x_{t})$ be a shortest horizontal path with $%
L(p)>L_{0}$, and let $p^{-}=(y_{0},y_{1},\dots ,y_{\ell })$ be its reduced
projected walk. We distinguish three cases, see Fig. \ref{Fig-classifc}.

\begin{enumerate}
\item[\textbf{Case A.}] $0\leq \ell \leq 1$.

\item[\textbf{Case B.}] $\ell \geq 2$ and there exist an index $i$ ($1\leq
i\leq \ell -1$) and integers $r<s$ such that 
\begin{equation}
x_{r}^{-}=x_{r+1}^{-}=\dots =x_{s}^{-}=y_{i}.  \label{rsi}
\end{equation}

\item[\textbf{Case C.}] $\ell \geq 2$ and for every index $i$ ($1\leq i\leq
\ell -1$) and every $k$ ($0\leq k\leq t-1$), 
\begin{equation}
x_{k}^{-}=x_{k+1}^{-}=y_{i}\quad \text{never holds}.  \label{never}
\end{equation}
\end{enumerate}

\begin{figure}[tbph]
\centering\includegraphics[width=1.0\textwidth]{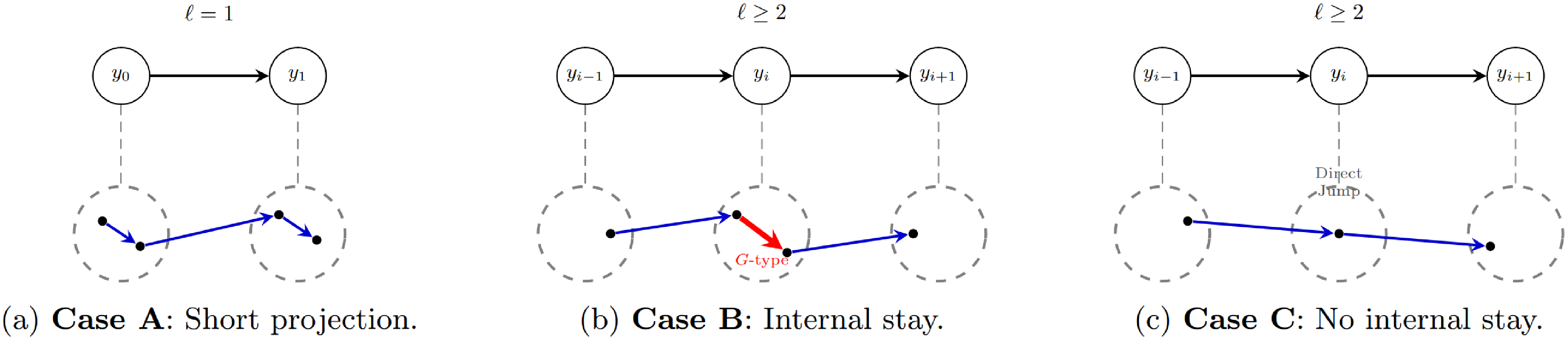} \vspace{%
-0.1cm}
\caption{Classification of a shortest horizontal path $p$.}
\label{Fig-classifc}
\end{figure}

\begin{claim}
\label{clm1}Suppose $K$ is nilpotent and the nilpotency index of $K$ is $I$.
Let $p$ be a shortest horizontal path in $\mathcal{U}$ with $L(p)>L_{0}$. If 
$p$ falls into Case A, then $M(p)=\boldsymbol{0}.$
\end{claim}

\begin{proof}
By definitions of $\Phi _{p}$ and of horizontal edges, for each edge $%
e_{i}=(x_{i-1},x_{i})$ of $p,$ the pair $(\Phi _{p}(x_{i-1}),\Phi
_{p}(x_{i}))$ is an edge of $\mathcal{G}$ with the same orientation as $%
e_{i} $. Since $\Sigma $ and $\Sigma ^{\prime }$ are disjoint and $p$ is a
shortest horizontal path, hence $\Phi (p)$ is also a shortest path in $%
\mathcal{G}$. In particular, the direction sequences coincide, i.e. $%
\mathfrak{d(}\Phi (p))=\mathfrak{d(}p).$ Note that $L(p)>L_{0}$, by the
definition of $L_{0}$ we have $M(p)=M(\Phi (p))=\boldsymbol{0}$.
\end{proof}

\begin{claim}
\label{clm2}Let $p\ $be a shortest horizontal path in $\mathcal{U}$ with $%
L(p)>L_{0}$. Under the same assumptions as in Theorem \ref{th-lac}, if $p$
falls into Case B, then $M(p)=\boldsymbol{0}.$
\end{claim}

\begin{proof}
Without loss of generality we can assume that $(x_{r},x_{r+1},\ldots ,x_{s})$
is the maximal contiguous subpath of $p$ such that all vertices have parent $%
y_{i}.$ Since $1\leq i\leq \ell -1,$ the vertices immediately preceding and
succeeding this block must have parents $y_{i-1}$\ and $y_{i+1}$ (i.e. $%
x_{r-1}^{-}=y_{i-1}$ and $x_{s+1}^{-}=y_{i+1}$)$.$ Consequently, the edges $%
(x_{r-1},x_{r})$ and $(x_{s},x_{s+1})$ are of $J$-type horizontal. It
follows from (\ref{S1S2}) that 
\begin{equation}
\varphi (x_{r}),\varphi (x_{s})\in S_{1}\cup S_{2}.  \label{sss}
\end{equation}%
Write $H:=G[U_{I-1}(\varphi (x_{r}))].$ It follows from (\ref{rsi}) that $%
(x_{j},x_{j+1})\in E_{h}$ is of $G$-type, and hence $(\varphi
(x_{j}),\varphi (x_{j+1}))\in E(G).$ Thus, sequence 
\begin{equation*}
\varphi (x_{r}),\varphi (x_{r+1}),\ldots \varphi (x_{s})
\end{equation*}%
gives a path from $\varphi (x_{r})$ to $\varphi (x_{s})$ on $\overrightarrow{%
G},$ which implies 
\begin{equation}
\varphi (x_{s})\in U_{s-r}(\varphi (x_{r})).  \label{s-r}
\end{equation}%
Since $\overrightarrow{G}$ satisfies the $(S_{1},S_{2},I-1)$-LAC and (\ref%
{sss})--(\ref{s-r}), we obtain that $H$ is an arborescence (or
anti-arborescence) of height exactly $I-1.$ By the construction of $G$-type
horizontal edges, the subpath 
\begin{equation*}
x_{r},x_{r+1},\ldots ,x_{r+I-1}
\end{equation*}%
has length $I-1$ and all its edges have the same direction letter $\sigma $.
From Claim \ref{clm0}, we obtain that the direction of $(x_{r-1},x_{r})$
coincides with $\sigma .$ Consequently, the subpath 
\begin{equation*}
x_{r-1},x_{r},x_{r+1},\ldots ,x_{r+I-1}
\end{equation*}%
has a directed sequence consisting of $I$ identical letter. Hence $S(p)\geq
I,$ and then by Corollary \ref{coro-N0-N1}, we have 
\begin{equation*}
M(p)=\boldsymbol{0}.
\end{equation*}
\end{proof}

For each direction $d\in \{0,1\}$ and $k\in \mathbb{N}\cup \{0\}$, we write $%
[d]_{k}$ for the word of length $k$ in $\{0,1\}^{k}$ whose letters are all
equal to $d$, i.e. 
\begin{equation*}
\lbrack d]_{k}=\underbrace{dd\cdots d}_{k\ \text{times}}.
\end{equation*}

\begin{lemma}
\label{le-kakb}Let $p$ be a shortest horizontal path in $\mathcal{U}$ with $%
L(p)>L_{0}$. Suppose the assumptions of Theorem\ \ref{th-lac} hold and $%
M(p)\neq \mathbf{0}$. Then there exist a shortest horizontal path $p^{\prime
}$ in $\mathcal{U}$, letters $a,b\in \{0,1\}$ and integers $k_{a},k_{b}\in 
\mathbb{N}\cup \{0\}$ such that 
\begin{equation*}
\mathfrak{d(}p)=[a]_{k_{a}}\mathfrak{d(}p^{\prime })[b]_{k_{b}},
\end{equation*}%
where the directed sequence $\mathfrak{d(}p^{\prime })$ starts with $a$ and
ends with $b$, $L(p^{\prime })\leq L_{0}$ and $M(p^{\prime })\neq 
\boldsymbol{0}$.
\end{lemma}

\begin{proof}
Claims \ref{clm1} and \ref{clm2} implies that $M(p)=\boldsymbol{0}$ whenever 
$p$ falls into Case A or Case B. Hence it suffices to consider Case C.\
Consequently, internal stay only occur within the $J_{1}(y_{0})$ and $%
J_{1}(y_{\ell }).$ Let $u_{0}$ denote the last child of $y_{0}$ visited by $%
p $ and $u_{\ell }$ denoe the first child of $y_{\ell }$ visited by $p$,
i.e.,%
\begin{equation*}
u_{0}=x_{i_{0}},\text{ }i_{0}=\max \{j:x_{j}^{-}=y_{0}\},\text{ }
\end{equation*}%
\begin{equation*}
u_{\ell }=x_{i_{\ell }},\text{ }i_{\ell }=\min \{j:x_{j}^{-}=y_{\ell }\}.
\end{equation*}%
Let $H_{0}=G[U_{I-1}(\varphi (u_{0}))],$ $H_{\ell }=G[U_{I-1}(\varphi
(u_{\ell }))],$ and set 
\begin{equation*}
a=\mathfrak{d(}y_{0},y_{1})=\mathfrak{d(}x_{i_{0}},x_{i_{0}+1}),
\end{equation*}%
\begin{equation*}
b=\mathfrak{d(}y_{\ell -1},y_{\ell })=\mathfrak{d(}x_{i_{\ell
}-1},x_{i_{\ell }}).
\end{equation*}

Note that $\varphi (u_{0}),\varphi (v_{\ell })\in S_{1}\cup S_{2}$, hence
the induced subgraphs $H_{0}$ and $H_{\ell }$ are (anti-)arborescences
rooted at $\varphi (u_{0})$ and $\varphi (v_{\ell })$, respectively. In
particular, 
\begin{equation*}
i_{0}<h(H_{0}),\ t-i_{\ell }<h(H_{\ell }).
\end{equation*}%
(Otherwise, if $i_{0}\geq h(H_{0})\ $or$\ t-i_{\ell }\geq h(H_{\ell }),$
then by Claim \ref{clm0} we have $S(p)\geq I.$ This implies $M(p)=0,$
contradicting the assumption that $M(p)\neq 0.$) Then we obtain that%
\begin{equation}
\mathfrak{d(}x_{0},x_{1},\cdots ,x_{i_{0}})=\mathfrak{d(}\varphi
(x_{0}),\varphi (x_{1}),\cdots ,\varphi (x_{i_{0}}))=[a]_{i_{0}}  \label{aaa}
\end{equation}%
and 
\begin{equation}
\mathfrak{d}(x_{i_{\ell }},x_{i_{\ell }+1},\cdots ,x_{t})=\mathfrak{d(}%
\varphi (x_{i_{\ell }}),\varphi (x_{i_{\ell }+1}),\cdots ,\varphi
(x_{t}))=[b]_{t-i_{\ell }}.  \label{bbb}
\end{equation}%
Using (\ref{never}) and (\ref{aaa})--(\ref{bbb}) we have 
\begin{equation}
\mathfrak{d(}p)=\mathfrak{d(}x_{0},x_{1},\cdots ,x_{i_{0}})\mathfrak{d(}%
p^{-})\mathfrak{d(}x_{i_{\ell }},x_{i_{\ell }+1},\cdots ,x_{t})=[a]_{i_{0}}%
\mathfrak{d(}p^{-})[b]_{t-i_{\ell }}  \label{decomp}
\end{equation}%
and $M(p^{-})\neq 0$ whenever $M(p)\neq 0$.

Since $\mathcal{G}$ is acyclic, by Lemma \ref{le-GPP} we obtain that $%
\mathcal{U}$ satisfies GPP. Hence $p^{-}$ is also a shortest horizontal
path. We now proceed by induction on $n$.

For the base case $n=2$, note that $p^{-}$ lies at level $1$, i.e. $p^{-}$
lies in $G_{1}\simeq G.$ \newline
If $i_{0}+(t-i_{\ell })<L(p)-L_{0}$, then 
\begin{equation*}
L(p^{-})=L(p)-i_{0}-(t-i_{\ell })>L_{0},
\end{equation*}%
so by the definition of $L_{0}$ we have $M(p^{-})=\boldsymbol{0},$ and hence 
$M(p)=\boldsymbol{0},$ which contradicts $M(p)\neq 0$. \newline
If $i_{0}+(t-i_{\ell })\geq L(p)-L_{0}$, then $L(p^{-})\leq L_{0}$. By (\ref%
{decomp}), we have 
\begin{equation*}
\mathfrak{d(}p)=[a]_{i_{0}}\mathfrak{d(}p^{-})[b]_{t-i_{\ell }}.
\end{equation*}

Assume the statement holds for some $n\geq 2$. Now we consider the shortest
horizontal path $p$ lying at level $n+1$ with $L(p)>L_{0}$ and $M(p)\neq 0$.
If $L(p^{-})\leq L_{0}$, then by (\ref{decomp}), the projected walk $p^{-}$
already provides the desired $p^{\prime }$.

Otherwise $L(p^{-})>L_{0}$. Since $p^{-}$ is a shortest horizontal path
lying at level $n$ and $M(p^{-})\neq 0$, the induction hypothesis yields a
horizontal path $p^{\prime }$ with $L(p^{\prime })\leq L_{0}$ and $%
M(p^{\prime })\neq 0$ such that 
\begin{equation*}
\mathfrak{d(}p^{-})=[a]_{k_{a}}\mathfrak{d(}p^{\prime })[b]_{k_{b}},
\end{equation*}%
where $\mathfrak{d(}p^{\prime })$ starts with $a$ and ends with $b$, and $%
k_{a},k_{b}\in \mathbb{N}\cup \{0\}$. Substituting this into (\ref{decomp})
gives 
\begin{equation*}
\mathfrak{d(}p)=[a]_{i_{0}}\mathfrak{d(}p^{-})[b]_{t-i_{\ell
}}=[a]_{i_{0}}[a]_{k_{a}}\mathfrak{d(}p^{\prime })[b]_{k_{b}}[b]_{t-i_{\ell
}}=[a]_{i_{0}+k_{a}}\mathfrak{d(}p^{\prime })[b]_{t-i_{\ell }+k_{b}},
\end{equation*}%
which is of the required form. This completes the inductive step.
\end{proof}

We are now ready to prove another main result of this subsection.

\begin{proof}[Proof of Theorem \protect\ref{th-lac}]
Let 
\begin{equation*}
L^{\ast }=L_{0}+2(I-2)+1.
\end{equation*}%
By Theorem \ref{th-equ}, it suffices to show that for every shortest
horizontal path $p$ in $\mathcal{U}$ with $L(p)\geq L^{\ast },$ we have $%
M(p)=\boldsymbol{0}$.

Let $p$ be a shortest horizontal path with $L(p)\geq L^{\ast }$ at level $n$%
. Then $p^{-}$ is a shortest horizontal path by Lemma \ref{le-GPP} and the
assumption that $\mathcal{G}$ is acyclic.

Assume, for contradiction, that $M(p)\neq \boldsymbol{0}$. According to the
Claims \ref{clm1} and \ref{clm2}, we may assume that $p$ falls into Case C.
Since $L(p)>L_{0}$ and $M(p)\neq \mathbf{0}$, Lemma~\ref{le-kakb} applies
and yields a horizontal path $p^{\prime }$ and integers $k_{a},k_{b}\in 
\mathbb{N}\cup \{0\}$ such that $L(p^{\prime })\leq L_{0},\ M(p^{\prime
})\neq 0,$ and 
\begin{equation}
\mathfrak{d(}p)=[a]_{k_{a}}\mathfrak{d(}p^{\prime })[b]_{k_{b}},  \label{d-p}
\end{equation}%
where $\mathfrak{d(}p^{\prime })$ starts with $a$ and ends with $b$.

Since $L(p)=k_{a}+L(p^{\prime })+k_{b}\geq L^{\ast }$ and $L(p^{\prime
})\leq L_{0}$, we have 
\begin{equation*}
k_{a}+k_{b}=L(p)-L(p^{\prime })\geq L^{\ast }-L_{0}=2(I-2)+1.
\end{equation*}%
Hence $\max \{k_{a},k_{b}\}\geq I-1$. If $k_{a}\geq I-1,$ then in (\ref{d-p}%
) the initial block $[a]_{k_{a}}$ followed by the first letter of $\mathfrak{%
d(}p^{\prime })$ (which is also $a$) yields a contiguous block of $a$'s of
length at least $k_{a}+1\geq I$ in $\mathfrak{d(}p)$. The case $k_{b}\geq
I-1 $ is analogous. Therefore the direction sequence $\mathfrak{d(}p)$
contains a contiguous block of length at least $I$, i.e. $S(p)\geq I.$ By
Corollary \ref{coro-N0-N1} this implies $M(p)=\boldsymbol{0},$ contradicting
our assumption $M(p)\neq \boldsymbol{0}$.
\end{proof}

\bigskip

\section{Applications and Examples}

In this section we apply the hyperbolicity criteria established in Section 4
to several basic classes of initial graphs $G$, including paths, cycles, and
complete graphs. These applications not only illustrate the scope of the
general results but also allow us to decide the hyperbolicity of the
substitution graphs appearing in Examples \ref{exa1}--\ref{exa4}.

Before discussing the hyperbolicity of the substitution graph in Example \ref%
{exa1}, we first present a more general result.

\begin{theorem}
Suppose $\overrightarrow{G}$ is a directed path from $1$ to $N$ and$\ 
\overrightarrow{E}(J)=\{(i,j^{\prime })\}.$ Then $\mathcal{U}$ is hyperbolic
if and only if 
\begin{equation}
i\neq j\text{ and either }i=N\text{ or }j=1.  \label{i=N,j=1}
\end{equation}
\end{theorem}

\begin{proof}
If $i\neq j,$ and $i=N$ or $j=1,$ then $\mathcal{G}$ is an arborescence, and
hence by Corollary \ref{cor-arbo} we obtain that $\mathcal{U}$ is hyperbolic.

Conversely, if $i\neq N$ and $j\neq 1,$ then $\deg _{\overrightarrow{G}%
}^{+}(i)\geq 1$ and $\deg _{\overrightarrow{G}}^{-}(j)\geq 1.$ Note that $G$
is connected and $\mathcal{G}$ is bipartite, it follows from Corollary \ref%
{coro-nonhyper} that $\mathcal{U}$ is not hyperbolic.

If $i=j,$ we claim that $\{D_{n}\}_{n\in \mathbb{N}}$ is unbounded, which
together with the non-nilpotency of $K$ and Proposition \ref{prop-nil}
implies that $\mathcal{U}$ is not hyperbolic.

Indeed, choose $l\in \{i-1,i+1\}\cap \Sigma $ (note that such an index $l$
exists as $N\geq 2$), then $d_{G}(l,i)=1$. For each $k\geq 1$ set 
\begin{equation*}
x_{k}:=l^{k},\ y_{k}:=i^{k}\in \Sigma _{k}.
\end{equation*}%
We prove by induction on $k$ that%
\begin{equation}
d_{h}(x_{k},y_{k})\geq k\quad \text{for all }k\geq 1.  \label{dist-k}
\end{equation}%
Since $d_{h}(x_{1},y_{1})=d_{G}(l,i)=1$, (\ref{dist-k}) holds for $k=1$.
Assume that (\ref{dist-k}) holds for some $k\geq 1$. Let $p\ $be a shortest
horizontal path from $x_{k+1}$ to $y_{k+1}$ with length $m$. Let $\ell
=L(p^{-}),$ then 
\begin{equation}
\ell \geq d_{h}(x_{k},y_{k}).  \label{x_k}
\end{equation}%
Because $E(\vec{J})=\{(i,i^{\prime })\}$, any $J$-type horizontal edge must
join two vertices whose last letters are both $i$. In particular, starting
from $x_{k+1}=l^{k+1}$, the first edge of $p$ must be a $G$-type edge that
changing the last letter from $l$ to $i$. Therefore this first edge does not
contribute to $p^{-}$, while each of the remaining $\ell -1$ edges
contributes at most one step in $p^{-}$. Thus 
\begin{equation}
\ell \leq m-1=d_{h}(x_{k+1},y_{k+1})-1.  \label{y_k}
\end{equation}%
Combining (\ref{x_k})--(\ref{y_k}) yields 
\begin{equation*}
d_{h}(x_{k+1},y_{k+1})\geq \ell +1\geq d_{h}(x_{k},y_{k})+1\geq k+1,
\end{equation*}%
hence (\ref{dist-k}) holds for $k+1$. By induction, (\ref{dist-k}) holds for
all $k\geq 1$. Since $x_{k},y_{k}\in \Sigma _{k}$, we have 
\begin{equation*}
D_{k}\geq d_{h}(x_{k},y_{k})\geq k\quad \text{for all }k\geq 1,
\end{equation*}%
and hence the sequence $(D_{n})_{n\geq 1}$ is unbounded.
\end{proof}

Example \ref{exa1} yields a hyperbolic substitution graph, as $%
\overrightarrow{G}$ is a directed path and it is easily verified that it
satisfies condition (\ref{i=N,j=1}).

Example \ref{exa2} satisfies the LAC condition, and $\mathcal{G}$ is
acyclic; therefore, Theorem \ref{th-lac} implies that its substitution graph
is hyperbolic.

\begin{proposition}
\label{prop-Cartesian}Let $G=K_{N}$ be the complete graph on $N$ vertices,
and let $E(\overrightarrow{J})=\{(i,i^{\prime }):i\in \Sigma \}.$ Then the
level-$n$ horizontal graph $G_{n}$ of the substitution graph $\mathcal{U}$\
is isomorphic to the $n$-fold Cartesian product of $G$, i.e. 
\begin{equation*}
G_{n}\simeq G^{n}=K_{N}^{n}\text{ for all }n\geq 1.
\end{equation*}%
Consequently, $\mathcal{U}$ is not hyperbolic.
\end{proposition}

\begin{proof}
We proceed by induction on $n$. For $n=1$, the level-$1$ horizontal graph $%
G_{1}$ coincides with $G=K_{N}$, and the identity mapping establishes an
isomorphism $G_{1}\simeq K_{N}$. Suppose that for some $n\geq 1$ there
exists an isomorphism $\phi _{n}:G_{n}\rightarrow K_{N}^{n}.$ We define a
mapping $\phi _{n+1}:G_{n+1}\rightarrow K_{N}^{n+1}$ by 
\begin{equation*}
\phi _{n+1}(v)=(\phi _{n}(v^{-}),\varphi (v)),\text{ }v\in \Sigma _{n}.
\end{equation*}%
By construction, two vertices $u,v\in \Sigma _{n}$ are adjacent if and only
if either $u^{-}=v^{-}$ and $(\varphi (u),\varphi (v))\in E(G)$ or $%
(u^{-},v^{-})\in E(G_{n})$ and $\varphi (u)=\varphi (v),$ which coincides
exactly with the adjacency relation in $K_{N}^{n}\times K_{N}=K_{N}^{n+1}.$
Hence, $\phi _{n+1}$ is an isomorphism.

Since $K_{N}^{n}$ have unbounded diameter as $n\rightarrow \infty $ and $K=K(%
\overrightarrow{J})$ is identity matrix (hence not nilpotent), it follows
from Proposition \ref{prop-nil} that $\mathcal{U}$ is not hyperbolic.
\end{proof}

As an immediate application of Proposition \ref{prop-Cartesian}, we obtain
that the substitution graph in Example \ref{exa3} is not hyperbolic.

As a direct consequence of Corollary \ref{coro-cycle}, the substitution
graph arising in Example \ref{exa4} is not hyperbolic.

\bigskip

\section{Acknowledgments}

This work was partially supported by the National Natural Science Foundation
of China (Nos. 62450131, 12301107 and 12271038) and the Shandong Provincial
Natural Science Foundation, China (No. ZR202209010046).

\end{document}